\documentclass[11pt]{amsart}

\usepackage{amsmath, amsthm, amssymb, amsfonts, enumerate}
\usepackage[colorlinks=true,linkcolor=blue,urlcolor=blue]{hyperref}
\usepackage{dsfont}
\usepackage{color}
\usepackage{geometry}
\usepackage{todonotes}

\newtheorem{theorem}{Theorem}[section]
\newtheorem{remark}[theorem]{Remark}
\newtheorem{assumption}[theorem]{Assumption}
\newtheorem{lemma}[theorem]{Lemma}
\newtheorem{proposition}[theorem]{Proposition}

\def \M{{\mathcal M}}
\def \E{\mathsf{E}}
\def \EE{\widehat{\mathsf{E}}}
\def \P{\mathsf{P}}
\def \PP{\widehat{\mathsf{P}}}
\def \R{\mathbb{R}}
\def \F{\mathbb{F}}

\def\X{\widehat{X}}

\def\WW{\widetilde{W}}

\def\rrho{\widehat{\rho}}
\def\PPsi{\widehat{\Psi}}

\def\ol{\overline}
\def\ul{\underline}

\usepackage{multirow}

\usepackage{setspace}
\usepackage{float}
\usepackage{verbatim}

\title[Debt Ratio Control with Regime Switching]{Optimal Control of Debt-to-GDP Ratio \\ in an $N$-state Regime Switching Economy}
	\author[Ferrari, Rodosthenous]{Giorgio Ferrari, Neofytos Rodosthenous}
\address{G.~Ferrari: Center for Mathematical Economics (IMW), Bielefeld University, Universit\"atsstrasse 25, 33615, Bielefeld, Germany}
\email{\href{mailto:giorgio.ferrari@uni-bielefeld.de}{giorgio.ferrari@uni-bielefeld.de}}
\address{N.~Rodosthenous: School of Mathematical Sciences, Queen Mary University of London, Mile End Road, London E1 4NS, United Kingdom}
\email{\href{mailto:n.rodosthenous@qmul.ac.uk}{n.rodosthenous@qmul.ac.uk}}
\date{\today}

\numberwithin{equation}{section}

\begin{document}

\begin{abstract} 
We solve an infinite time-horizon bounded-variation stochastic control problem with regime switching between $N$ states. This is motivated by the problem of a government that wants to control the country's debt-to-GDP (gross domestic product) ratio. In our formulation, the debt-to-GDP ratio evolves stochastically in continuous time, and its drift -- given by the interest rate on government debt, net of the growth rate of GDP -- is affected by an exogenous macroeconomic risk process modelled by a continuous-time Markov chain with $N$ states. The government can act on the public debt by increasing or decreasing its level, and it aims at minimising a net expected regime-dependent cost functional. Without relying on a guess-and-verify approach, but performing a direct probabilistic study, we show that it is optimal to keep the debt-to-GDP ratio in an interval, whose boundaries depend on the states of the risk process. These boundaries are given through a zero-sum optimal stopping game with regime switching with $N$ states and are characterised through a system of nonlinear algebraic equations with constraints. To the best of our knowledge, such a result appears here for the first time. Finally, we put in practice our methodology in a case study of a Markov chain with $N=2$ states; we provide a thorough analysis and we complement our theoretical results by a detailed numerical study on the sensitivity of the optimal debt ratio management policy with respect to the problem's parameters.
\end{abstract}

\maketitle


\smallskip

{\textbf{MSC2010 subject classification}}: 93E20, 60G40, 60J60, 60J27, 91B64.

\section{Introduction}

It has been observed that during the financial crisis that started in 2007, debt-to-GDP ratio (also called the ``debt ratio") exploded from an average of 53\% to circa 80\% in many countries. Ever since, there has been a huge debate in the economic and political community on the sustainability of public debt. Using different statistical and methodological approaches, many researchers conclude that high government debt has negative economic and financial effects, as it makes the economy less resilient to macroeconomic shocks (e.g.\ sovereign default risks and liquidity shocks), and poses limits to the adoption of counter-cyclical fiscal policies (see \cite{RRR12}, among many others).
The common view derived from the empirical evidence is that, from the perspective of a government's general economic planning, it is important to reduce high levels of debt ratio in order to maintain fiscal sustainability and support stronger fundamentals.
However, in \cite{IMF} researchers from the International Monetary Fund also suggest that reducing the debt ratio might not be always the most sensible approach. The conclusion seems to apply in particular to those countries enjoying sufficient ``fiscal space''\footnote{This is the  distance between the government's debt ratio and an ``upper limit'', calculated by the Moody's ratings agency, beyond which the government should reduce debt in order to avoid default.}, like U.S.A., Germany and the U.K.. 
When deciding their economic planning, governments are presented with two questions: \emph{How much is too much?} and \emph{How low is too low?}.
In this paper, we propose a mathematical formulation of the optimal debt ratio's management problem faced by a government that addresses both of these questions.

In our model, the GDP of the country is a geometric Brownian motion with growth rate $g$ and volatility (per unit of GDP) $\sigma$. The real debt evolves exponentially with rate $r+\lambda_{Y_t}$, which is the interest rate on debt that the government pays at time $t$. This consists of a fixed deterministic component $r$ and a stochastic, time-varying component $\lambda_{Y_t}$. As a matter of fact, this is a generalisation of the standard exponential evolution of real debt with constant rate that one can find in classical textbooks of macroeconomics (see \cite{BlanchardFischer}, among others). The stochastic, time-varying component of the interest rate is driven by a continuous-time Markov chain $Y$ with $N$ states, modelling market factors that are not under the control of the government. In this sense, $\lambda_{Y_t}$ is the additional interest that the government pays on debt at time $t$, e.g.\ due to a change of the credit rating of the country, or to a mass sell-off of government bonds. As a result, in absence of any intervention, the debt-to-GDP ratio evolves stochastically following geometric dynamics with regime switching in the drift $r+\lambda_{Y_t}-g$.

When in debt, the government incurs an instantaneous cost which may be interpreted as an opportunity cost resulting, e.g., from private investments crowding out, less room for financing public investments, and from a tendency to suffer low subsequent growth (see \cite{RRR12}, among others, for empirical studies). 
We allow this cost to depend on the current economic regime $Y$. 
The government may intervene in order to decrease or increase the level of the debt ratio, and we assume that these policies have an instantaneous effect. Consequently, the cumulative amount of debt ratio's increase and decrease are the government's control variables. Any decrease of the debt ratio by the government results in proportional costs, whereas any increase results in proportional benefits. We further assume that the government discounts costs and benefits at a stochastic time-varying rate modulated by the changes in the economic regime. The objective of the government is to minimise the total expected discounted costs incurred by debt and the cost of decreasing the debt ratio, net of the benefits arising from an increase of the latter by the government.

The mathematical formulation associated with the above problem is that of a bounded-variation stochastic control problem, in which the state process is a linearly controlled geometric Brownian motion with regime switching and the cost functional is regime dependent as well. This is due to the $N$-state Markov chain $Y$ modelling the macroeconomic conditions. 
We succeed in determining the explicit solution to this problem. 
To the best of our knowledge, this is the first paper which completely solves a singular stochastic control problem with: (i) regime switching between an arbitrary number $N \geq 2$ of states and (ii) controls of bounded-variation.

We solve this problem \emph{without} relying on a classical \emph{guess-and-verify} approach. Indeed, if we attempt to follow such an approach, we should solve a system of $N$ coupled ordinary differential equations with gradient constraints (the coupling is through the transition rates of the Markov chain $Y$), and then verify that the obtained solution satisfies the dynamic programming equation which takes the form of a variational inequality. Given the complexity of the problem under consideration, this approach seems not to be feasible. In fact, even in the particular example with $N=2$ regimes addressed in Section \ref{sec:casestudy}, the guess-and-verify approach would require proving existence and uniqueness of a quadruple solving a highly nonlinear system of four algebraic equations with constraints (see \eqref{System-1}--\eqref{System-4} with \eqref{bound-1}--\eqref{bound-2} below). Obviously, the complexity increases with $N$ (see Remark \ref{rem:complex}).

Instead, here we tackle the problem via a \emph{direct probabilistic approach}, by relating the bounded-variation stochastic control problem to a zero-sum game of optimal stopping (Dynkin game) with regime switching. 
Such a connection to a Dynkin game has two main advantages. 
Firstly, it provides the geometry of the state space in terms of regions where it is optimal to intervene and wait. 
Secondly, it allows to achieve the regularity of the control problem's value function $V$ needed for the characterisation of the optimal control policy. 
Our analysis begins by first proving an abstract existence and uniqueness result for the optimal debt-management policy, upon relying on a suitable application of Koml\'os' theorem (see also \cite{DeAFe14} and \cite{KaratzasWang}). Using this result, we apply Theorems 3.1 and 3.2 of \cite{KaratzasWang}, and provide the form of a Dynkin game with regime switching, whose value $v$ coincides with the first derivative of $V$. We then study the Dynkin game by employing mostly probabilistic arguments, and we prove the structure of its saddle point. This consists of a couple of entry times to two connected regions (the so-called ``stopping regions'') whose boundaries $a$ and $b$ depend on the current regime of the Markov chain $Y$. For any such regime $i$, we then prove that $v$ is everywhere continuously differentiable, thus implying the well-known smooth-fit condition of $v$ at the boundaries of the stopping regions. Such a regularity of $v$, in turn, immediately gives that $V$ is $C^2$ for any regime $i$. Hence, through this direct approach, we manage to prove that $V$ is a classical solution to the corresponding dynamic programming equation, which we use to provide the structure of an optimal control rule. At any time, this prescribes to keep the (optimally) controlled debt ratio process inside the interval $[a(Y_t),b(Y_t)]$, either in a minimal way (i.e.\ according to a Skorokhod reflection) if it is already inside, or with an immediate jump, if it suddenly goes outside (i.e.\ according to a lump-sum increase/decrease). Thus, these two levels defining the interval, trigger the timing at which the government should optimally intervene to either increase or decrease the debt ratio.
It is worth noticing that the aforementioned methodology can also be applied to solve other singular or bounded-variation stochastic control problems under regime switching with such an arbitrary number $N\geq 2$ of states. These could be natural directions for future research.

In order to prove the existence of an optimal control policy we need to impose a condition on the marginal cost and benefit of increasing and decreasing the debt ratio, respectively. 
Interestingly, in Section \ref{sec:geometry}, we show that this condition also plays a fundamental role in establishing an ordering of the optimal stopping boundaries $a(i)$ and $b(i)$ across the $N$ different regimes $i$. In particular, this result can be exploited to determine the explicit equations that the optimal boundaries $a$ and $b$ necessarily satisfy. These equations follow from the $C^1$-property of $v$ previously proved. We put in practice our methodology in Section \ref{sec:casestudy} in a case study of $N=2$ regimes. To the best of our knowledge, even the study of the case with $N=2$ regimes appears in this paper for the first time. Finally, we complement our theoretical results by a detailed numerical study on the sensitivity of the optimal debt ratio management policy with respect to the problem's parameters.

Our paper is placed among those few works employing continuous-time singular stochastic control methods for public debt management. In \cite{Cad} and \cite{Cad2}, the debt ratio evolves as a linearly controlled one-dimensional geometric Brownian motion and the government can only reduce its level through singular controls and bounded-velocity controls, respectively. 
The objective is to minimise the total expected costs arising from having debt and intervening on it. Instead, in our model, the government can both reduce and increase the debt ratio, and the dynamics of the latter is affected by two sources of uncertainty: a Brownian motion and a continuous-time Markov chain. In \cite{Ferrari}, the problem is again to only optimally reduce the debt ratio, but in that case the government takes into consideration the evolution of the inflation rate of the country. The latter evolves as an uncontrolled diffusion process which makes the problem a fully two-dimensional singular stochastic control problem. This clearly leads to a completely different mathematical treatment than this paper. 
In \cite{CCF}, a partially informed government on the underlying business conditions, once again only reduces the debt ratio. By adopting filtering techniques, the government's optimal control problem is reduced to one under full information, and then solved in a case study.

Also the literature on singular stochastic control problems with regime switching is still limited, and most of the papers deal only with Markov chains with $N=2$ states and with monotone controls. We refer, e.g., to \cite{Pistorius} and \cite{CadSot} where the optimal dividend problem of actuarial science is formulated as a one-dimensional monotone follower problem; to \cite{Guoetal} for an irreversible investment problem; to the recent \cite{FerrariYang} for an optimal extraction problem. 
In this paper, we provide the complete solution to a singular stochastic control problem under regime switching with $N\geq 2$ states, where the control processes are not monotone but have paths of bounded variation.

The rest of the paper is organised as follows. In Section \ref{setting}, we set up the model and provide the control problem formulation of the government. In Section \ref{sec:existence}, we prove the existence and uniqueness of the optimal debt ratio management policy, and we introduce the associated Dynkin game. In Section \ref{sec:OSgame}, we study the Dynkin game and we characterise its saddle point. These results are then used in Section \ref{sec:OCrule} to construct the optimal debt ratio management policy. The geometry of the problem's state space is studied in Section \ref{sec:geometry}, while a case study with $N=2$ regimes is then considered in Section \ref{sec:casestudy}. We also provide a detailed comparative statics analysis (see Section \ref{sec:compstat}) and comparison with the non-regime-switching case (see Section \ref{sec:comparison}).

\section{Setting and Problem Formulation}
\label{setting}

Let $(\Omega, \mathcal{F}, \P)$ be a complete probability space rich enough to accommodate a one-dimensional Brownian motion $W:=(W_t)_{t\geq 0}$ and a continuous-time Markov chain $Y:=(Y_t)_{ t\geq 0}$. To be more precise, $Y$ is such that for all $t\geq 0$, $Y_{t} \in \mathcal{M}:=\{1,2,\dots,N\}$ for some $N\geq 2$, and it has an irreducible generator matrix $Q:=(q_{ij})_{1 \leq i,j \leq N}$ with
\begin{equation*}
\P(Y_{t+\Delta t}=j\,|\,Y_{t}=i, Y_s, s\leq t):=\left\{
\begin{array}{lr}
\displaystyle q_{ij}\Delta t + o(\Delta t)\,\,\qquad\,\,\, \mbox{if $j\neq i$}\\[+14pt]
\displaystyle 1 + q_{ii}\Delta t + o(\Delta t)\,\,\quad \mbox{if $j=i$}.
\end{array}
\right.
\end{equation*}
Here $q_{ij}\geq 0$ for $(i,j) \in \mathcal{M} \times \mathcal{M}$ with $j\neq i$, and $q_{ii}=-\sum_{j \neq i}q_{ij}<0$ for each $i\in \mathcal{M}$. The Markov chain $Y$ jumps between the states at exponentially distributed random times, and the constant $q_{ij}$ gives the rate of jumping from state $i$ to $j$. We take $Y$ independent of $W$, and denote by $\mathbb{F}:=\{\mathcal{F}_t, t\geq 0\}$ the filtration jointly generated by $W$ and $Y$, and as usual, augmented by $\P$-null sets.

We assume that in absence of any intervention by the government, the debt-to-GDP ratio evolves according to the stochastic differential equation (SDE)
\begin{equation}
\label{freeGBM}
dX^0_t=\big(r + \lambda_{Y_t} - g\big)X^0_t dt + \sigma X^0_t dW_t, \quad t>0, \qquad X^0_0=x > 0.
\end{equation}
These dynamics might be seen as a stochastic version of the one proposed in classical macroeconomic textbooks, see e.g. \cite{BlanchardFischer}.
Here $g \in \R$ is the growth rate of the GDP, whereas $r+\lambda_{Y_t}$ is the interest rate on government debt. This interest rate consists of a basis fixed component $r>0$, and of a time-varying stochastic component $\lambda_{Y_t}$ which represents the additional interest rate that the country has to pay at time $t$ when the macroeconomic conditions are in state $Y_t \in \mathcal{M}$. 
\begin{assumption} 
\label{lY}
Without loss of generality, we assume that $\lambda_1 \geq \lambda_2 \geq \dots \geq \lambda_N$, hence $\lambda_{Y_t} \in [\lambda_N,\lambda_1]$, $\P$-a.s.\ for all $t\geq0$.
\end{assumption}

In the following we will often denote by $X^{x,i,0}$ the unique strong solution to \eqref{freeGBM} starting at time zero from level $x>0$ when $Y_0=i \in \mathcal{M}$; that is,
\begin{equation}
\label{solGBM}
X^{x,i,0}_t = \displaystyle x e^{(r - g - \frac{1}{2}\sigma^2)t + \int_0^t \lambda_{Y^i_s}ds + \sigma W_t}, \qquad t\geq 0. 
\end{equation}
We also denote by $Y^i_t$ the Markov chain $Y_t$ started from state $i\in \mathcal{M}$ at initial time. 

\begin{remark}
\label{rem:debtdyn}
Dynamics \eqref{freeGBM} can be justified in the following way. In absence of any intervention by the government, the nominal debt $D_t$ grows at time $t\geq0$ at rate $r + \lambda_{Y_t}$; i.e., $dD_t = (r + \lambda_{Y_t})D_t dt$. Assuming that the GDP, $\psi$, evolves stochastically as  
$$d\psi_t=g\psi_t dt + \sigma \psi_t dB_t,$$
for some Brownian motion $B$, an application of It\^o's formula and a change of measure shows that $X^0:=D/\psi$ follows the geometric dynamics \eqref{freeGBM}.
\end{remark}

The government can increase or decrease the current level of the debt-to-GDP ratio by, e.g., making investments on infrastructures or imposing austerity policies in the form of spending cuts, respectively.
Denoting by $\eta_t$ the cumulative amount, e.g., of spending cuts made up to time $t\geq0$ in order to reduce the debt-to-GDP ratio, and by $\xi_t$ the cumulative amount, e.g, of investments made up to time $t\geq0$, the dynamics of the adjusted debt-to-GDP ratio read as
\begin{equation}
\label{DGDP}
dX_t=\big(r + \lambda_{Y_t} - g\big)X_t dt + \sigma X_t dW_t + d\xi_t - d\eta_t, \quad t>0, \qquad X_0=x \in \mathbb{R_+}.
\end{equation}

Given that $\xi$ and $\eta$ represent the cumulative interventions, it is natural to model them as nondecreasing stochastic processes, adapted with respect to the available flow of information $\mathbb{F}$. Hence we take $\xi$ and $\eta$ in the set
\begin{align*}
\mathcal{U} := & \{ \vartheta:\Omega \times \R_+ \to \R_+, \mbox{ $\F$-adapted and such that } t \mapsto \vartheta_t \mbox{ is a.s. } \mbox{nondecreasing and left-continuous}\}. \nonumber 
\end{align*}
In the following, we set $\vartheta_0 = 0$ a.s.\ for any $\vartheta \in \mathcal{U}$. We suppose that the government cannot make at the same instant in time interventions to increase and decrease the debt ratio; i.e., we assume that the 
(random) measures $d\xi_\cdot$ and $d\eta_\cdot$ on $\R_+$ induced by the nondecreasing processes $\xi$ and $\eta$, respectively, have disjoint supports.
We then denote by $\varphi$ the process belonging to
\begin{align*}
\mathcal{V} := & \{\zeta:\Omega \times \R_+ \to \R, \mbox{ $\F$-adapted and such that } t \mapsto \zeta_t \mbox{ is a.s.} \\
& \hspace{0.5cm} \mbox{(locally) of bounded variation, left-continuous and } \zeta_0 = 0 \}, \nonumber
\end{align*}
whose unique minimal decomposition is given by the two nondecreasing processes $\xi$ and $\eta$; that is, $\varphi_t = \xi_t - \eta_t$, for all $t \geq 0$.

For any $\varphi = \xi - \eta \in \mathcal{V}$, equation \eqref{DGDP} admits the unique strong solution 
\begin{equation}
\label{solDTGDP}
X^{x,i,\varphi}_t=X^{1,i,0}_t\bigg[x + \int_{[0,t)}\frac{d\xi_s}{X^{1,i,0}_s} - \int_{[0,t)}\frac{d\eta_s}{X^{1,i,0}_s}\bigg], \qquad t\geq 0,
\end{equation}
where we have stressed the dependency on $\varphi\in \mathcal{V}$ and on the initial datum $(x,i)\in \R_+ \times \mathcal{M}$ by writing $X^{x,i,\varphi}$. Here, $X^{x,i,0}$ is as in \eqref{solGBM}, and it is the unique strong solution to \eqref{DGDP} when $\xi=\eta\equiv0$ and therefore $\varphi \equiv 0$.

Having a debt ratio level $X_t$ under the state $Y_t$ at time $t\geq 0$, the government incurs an instantaneous cost $h(X_t, Y_t)$. This may be interpreted as an opportunity cost that depends on the current macroeconomic conditions and results from private investments' crowding out, less room for financing public investments, and from a tendency to suffer low subsequent growth.

In the rest of the paper, we set $g_x(x,i):=\frac{\partial g}{\partial x}(x,i)$ for any differentiable function $g : \R \times \M \to \R$, and we make the following standing assumption on the running cost function $h:\R \times \M \mapsto \R_+$.
\begin{assumption}
\label{ass:h}
For any $i \in \M$, we have 
\begin{itemize}
\item[(i)] $x \mapsto h(x,i)$ is strictly convex, continuously differentiable and increasing on $[0,\infty)$, and it is such that $h(x,i)=0$ for any $x\leq0$;
\item[(ii)] the derivative $h_x$ of $h$ satisfies $h_x(0,i)=0$ and $\lim_{x \rightarrow \infty}h_x(x,i)=+\infty$;
\item[(iii)] there exists $m > 1$, $K_1>0$, $K_2>0$ and $K_3>0$ such that $$h(x,i) \leq K_1(1 + |x|^{m}) \quad \text{and} \quad |h_x(x,i)| \leq K_2(1 + |x|^{m-1}), \quad x \in \mathbb{R},$$ and 
$$|h_x(x,i)-h_x(y,i)| \leq K_3 |x - y| (1 + |x|^{(m-2)^+}), \quad (x,y)\in \mathbb{R}^2;$$
\item[(iv)] $h(\cdot,i)$ has finite Legendre transform on $(0,\infty)$; that is, for all $p > 0$ we have $\sup_{x\in \R_+}\big(p x -h(x,i)\big) < \infty.$
\end{itemize}
\end{assumption}

\begin{remark}
\label{rem:h}
It is worth noticing that a cost function of the form $h(x,i) = \frac{\varkappa(i)}{2}x^2$ for any  $(x,i) \in \R_+ \times \M$ and $h(x,i)=0$ for any $(x,i) \in \R_- \times \M$ satisfies Assumption \ref{ass:h}. 
Moreover, the assumption $h(0,i)=0$ is without loss of generality, since if $h(0,i) = h_o(i) > 0$ then one can always set $\widetilde{h}(x,i):= h(x,i) - h_o(i)$ and write $h(x,i) = \widetilde{h}(x,i) + h_o(i)$, so that the optimisation problem (cf.\ \eqref{eq:valueOC} below) remains unchanged up to an additive constant. Notice that such a requirement, together with $h_x(0,i)=0$, implies that any infinitesimal amount of debt does not generate holding costs for the country; indeed, $h(\varepsilon,i)\approx h_x(0,i)\varepsilon=0$.
\end{remark}

Whenever the government decides to reduce the level of debt ratio, it incurs an intervention cost that is proportional to the amount of debt reduction (see also \cite{Cad} and \cite{Ferrari}). This might be seen as a measure of the social and financial consequences deriving from a debt-reduction policy, and the associated marginal cost $c_1>0$ allows to express it in monetary terms.
On the other hand, the government can increase the current level of debt ratio (e.g.\ through investments in infrastructure, healthcare, education and research, etc.), and we assume that this has a positive social and financial effect, thus overall reduces the total expected ``costs'' of the government. The marginal benefit of increasing the debt ratio is a constant $c_2>0$. 

We further assume that the government discounts at a strictly positive time-varying stochastic rate $\rho_{Y_t} \in [\ul\rho, \ol\rho]$, when the macroeconomic conditions are in state $Y_t \in \mathcal{M}$ at time $t \geq 0$. 
Then, the total expected cost functional, net of investment benefits, is
\begin{equation}
\label{eq:J}
\mathcal{J}_{x,i}(\varphi):=\E_{(x,i)}\bigg[\int_0^{\infty} e^{-\int_0^t \rho_{Y_s}ds} h(X^{\varphi}_t, Y_t) dt + c_1 \int_0^{\infty}  e^{-\int_0^t \rho_{Y_s}ds} d\eta_t - c_2 \int_0^{\infty} e^{-\int_0^t \rho_{Y_s}ds} d\xi_t\bigg],
\end{equation}
where, for any $(x,i) \in \mathcal{O}:=\R_+ \times \mathcal{M}$, $\E_{(x,i)}$ denotes the expectation under the measure $\P_{(x,i)}(\,\cdot\,):=\P(\,\cdot\,|X^{\varphi}_0=x,Y_0=i)$. In the following we will equivalently write $\E[f(X^{x,i,\varphi}_t, Y_t^i)]=\E_{(x,i)}[f(X^{\varphi}_t, Y_t)]$, for any $t\geq 0$ and Borel-measurable function $f:\mathbb{R} \times \M \to \mathbb{R}$ such that the previous expectation is finite.
Hereafter, we use the notation $\int_0^{t} (\,\cdot\,) d\vartheta_s = \int_{[0,t)} (\,\cdot\,) d\vartheta_s$, for $\vartheta \in \{\xi,\eta\}$ and any $t \in [0,\infty]$.

For any given initial value of the debt ratio $x\geq0$ and of the state of the economy $i\in \mathcal{M}$, we assume that the government will not use a debt ratio management policy leading to infinite cost/benefit of interventions, and given that the debt ratio level is always a positive number, the government picks its debt ratio management policy $\varphi$ in the set 
\begin{align}
\label{setA}
\displaystyle \mathcal{A}(x,i):=\Big\{\varphi \in \mathcal{V}:\,\E_{(x,i)}\bigg[\int_0^{\infty} e^{-\int_0^t \rho_{Y_s}ds} \big(d\eta_t + d\xi_t\big)\bigg]< \infty\,\,\text{and}\,\,X^{x,i,\varphi}_t\geq 0\quad \P\otimes dt-\mbox{a.e.}\Big\}.
\end{align}
The government's aim is therefore to solve
\begin{equation}
\label{eq:valueOC}
V(x,i):=\inf_{\varphi \in \mathcal{A}(x,i)}\mathcal{J}_{x,i}(\varphi), \qquad (x,i) \in \mathcal{O}.
\end{equation}
We will refer to $V$ as the value function, and any debt ratio management policy belonging to $\mathcal{A}$ will be called admissible.

The following assumption on the model's parameters will hold true in the rest of this paper.
\begin{assumption}
\label{ass:c12}
The model's parameters satisfy
$$c_1 ({\ul\rho - r + g - \lambda_1}) > c_2 ({\ol\rho - r + g - \lambda_N})  $$
\end{assumption}

\noindent Since $\lambda_N \leq \lambda_1$ and $\ul\rho < \ol\rho$, Assumption \ref{ass:c12} in particular implies the condition $c_1>c_2$. This is typically assumed in the literature on bounded-variation stochastic control problems in order to ensure well-posedness of the optimisation problem (see, e.g., \cite{DeAFe14} and \cite{GuoPham}) and to avoid arbitrage opportunities.
Assumption \ref{ass:c12} will play a central role in the proof of existence of an optimal debt ratio management policy for problem \eqref{eq:valueOC} (see the proof of Lemma \ref{lem:ui} below). It is also worth noticing that Assumption \ref{ass:c12} will have important implications on the geometry of the state space (see Proposition \ref{prop:structure} below).

\section{On the Existence of the Optimal Debt Ratio Management Policy}
\label{sec:existence}

In this section we prove some preliminary properties of the value function, the existence and uniqueness of an optimal debt ratio management policy for problem \eqref{eq:valueOC}, and its relation to a zero-sum game of optimal stopping (Dynkin game).

We start with the following result, whose proof is standard and therefore omitted.

\begin{proposition}
\label{prop:preliminaryV}
The value function $V$ of \eqref{eq:valueOC} is such that $x \mapsto V(x,i)$ is convex on $\R_+$ for any $i \in \mathcal{M}$. Moreover, 
$V(x,i) \leq c_1 x$ for all $(x,i) \in \mathcal{O}$. 
\end{proposition}

To take care of the infinite time-horizon of our problem we need the following assumption, which will also hold throughout the rest of this paper.
\begin{assumption}
\label{ass:rho}
Recall $m$ from Assumption \ref{ass:h}. The model's parameters satisfy
$$\ul\rho > \Big((r - g + \lambda_1) \vee \big(m(r - g + \lambda_1) + \frac{\sigma^2}{2}m(m-1)\big)\Big)^{+}.$$
\end{assumption}
Assumption \ref{ass:rho} may be justified by noting that the government, which runs only for a limited amount of years,  is more concerned about the present than the future, and therefore discounts future costs and benefits at a sufficiently large rate. Moreover, a combination of the condition $\ul\rho > (m(r - g + \lambda_1) + \frac{\sigma^2}{2}m(m-1))^+$ with Assumption \ref{ass:h}-(iii), ensures that the trivial admissible policy ``do not intervene at all on the debt ratio'' yields a finite expected cost, even if it is not necessarily the minimal one. 

Notice that setting 
\begin{equation}
\label{controlsbar}
\overline{\xi}_t:= \int_{0}^t \frac{d\xi_s}{X^{1,i,0}_s} \qquad\mbox{and}\qquad  \overline{\eta}_t:= \int_{0}^t \frac{d\eta_s}{X^{1,i,0}_s},\qquad \overline{\xi}_0=0=\overline{\eta}_0,
\end{equation}
and $\overline{\varphi}:=\overline{\xi} - \overline{\eta}$, the solution to \eqref{solDTGDP} rewrites as
\begin{equation}
\label{Xbar}
X^{x,i,\varphi}_t = X^{1,i,0}_t\big[x + \overline{\xi}_t  - \overline{\eta}_t\big], \qquad t\geq 0.
\end{equation}
The quantities $d\overline{\xi}_t$ and $d\overline{\eta}_t$ are the sizes of interventions made at time $t\geq0$, per unit of debt ratio in absence of any intervention.

Then, by defining $\overline{\mathcal{A}}$, for any $(x,i) \in \mathcal{O}$, as 
$$\displaystyle \overline{\mathcal{A}}(x,i):=\Big\{\overline{\varphi} \in \mathcal{V}:\,\E\bigg[\int_0^{\infty} e^{-\int_0^t \rho_{Y_s^i}ds} X^{1,i,0}_t \big(d\overline{\eta}_t + d\overline{\xi}_t\big)\bigg]< \infty\,\,\text{and}\,\,x + \overline{\xi}_t  - \overline{\eta}_t \geq 0\quad \P\otimes dt-\mbox{a.e.}\Big\},$$
it is easy to see that the mapping $\mathcal{A}(x,i) \ni \varphi \mapsto \overline{\varphi} \in \overline{\mathcal{A}}(x,i)$ is one-to-one and onto, and one can also write for any $(x,i) \in \mathcal{O}$
\begin{align}
\label{reprV}
V(x,i) =\inf_{\overline{\varphi} \in \overline{\mathcal{A}}(x,i)} \E\bigg[&\int_0^{\infty} e^{-\int_0^t \rho_{Y_s^i}ds} h\big(X^{1,i,0}_t\big[x + \overline{\xi}_t  - \overline{\eta}_t\big], Y_t^i\big) dt \nonumber\\
&+ c_1 \int_0^{\infty} e^{-\int_0^t \rho_{Y_s^i}ds} X^{1,i,0}_t d\overline{\eta}_t 
- c_2 \int_0^{\infty} e^{-\int_0^t \rho_{Y_s^i}ds} X^{1,i,0}_t d\overline{\xi}_t\bigg].
\end{align}
The definitions of $\overline{\xi}$ and $\overline{\eta}$ in \eqref{controlsbar} will be used in the proof of the next result.

\begin{lemma}
\label{lem:ui}
Let $(x,i) \in \mathcal{O}$ be arbitrary but fixed, and let $(\varphi^n)_{n\in \mathbb{N}}:=(\xi^n, \eta^n)_{n\in \mathbb{N}}$ be a minimising sequence for problem \eqref{eq:valueOC} (equivalently, \eqref{reprV}). Then
\begin{equation}
\label{eq:ui}
\sup_{n \in \mathbb{N}}\E_{(x,i)}\bigg[\int_0^{\infty} e^{-\int_0^t \rho_{Y_s}ds} d\eta^n_t + \int_0^{\infty} e^{-\int_0^t \rho_{Y_s}ds} d\xi^n_t\bigg] < \infty.
\end{equation}
\end{lemma}
\begin{proof}
Let $(x,i) \in \mathcal{O}$ be given and fixed, and let $(\varphi^n)_{n\in \mathbb{N}}:=(\xi^n, \eta^n)_{n\in \mathbb{N}}$ be a minimising sequence for problem \eqref{eq:valueOC} (equivalently, \eqref{reprV}). Without loss of generality, we can take $(\varphi^n)_{n\in \mathbb{N}}$ such that
$$1 + V(x,i) \geq \mathcal{J}_{x,i}(\varphi^n), \qquad \mbox{for any}\,\,n,$$
and then recalling that $V(x,i) \leq c_1 x$ due to Proposition \ref{prop:preliminaryV}, it follows  from \eqref{eq:J} and \eqref{Xbar} that 
\begin{align}
\label{ui1}
 1 + {c_1} x \geq 1 + V(x,i) \geq \mathcal{J}_{x,i}(\varphi^n) = &\,\E\bigg[\int_0^{\infty} e^{-\int_0^t \rho_{Y_s^i}ds} h\big(X^{1,i,0}_t\big[x + \overline{\xi}^n_t  - \overline{\eta}^n_t\big], Y_t^i \big) dt\bigg] \nonumber \\
&+ \E_{(x,i)}\bigg[c_1 \int_0^{\infty} e^{-\int_0^t \rho_{Y_s}ds} d{\eta}^n_t - c_2 \int_0^{\infty} e^{-\int_0^t \rho_{Y_s}ds} d{\xi}^n_t\bigg].
\end{align}
By Assumption \ref{ass:h}-(iv), for any $\varepsilon > 0$ there exists $\kappa_{\varepsilon}>0$ such that $h(x,i)\geq \varepsilon x - \kappa_{\varepsilon}$ for any $(x,i) \in \mathbb{R}_+\times \mathcal{M}$. 
Taking this into account together with the monotonicity of $h(\cdot,i)$ in Assumption \ref{ass:h}-(i), \eqref{solGBM} and the positivity of $xX^{1,i,0}$, we can therefore continue from \eqref{ui1} by writing 
\begin{align}
\label{ui2}
1 + {c_1} x 
\geq &-\frac{\kappa_{\varepsilon}}{\ul\rho} + \varepsilon \E\bigg[\int_0^{\infty} e^{-\int_0^t \rho_{Y_s^i}ds} X^{1,i,0}_t \big(\overline{\xi}^n_t - \overline{\eta}^n_t\big) dt\bigg] \nonumber \\
&+ \E_{(x,i)}\bigg [ {c_1} \int_0^{\infty} e^{-\int_0^t \rho_{Y_s}ds} d{\eta}^n_t - {c_2} \int_0^{\infty} e^{-\int_0^t \rho_{Y_s}ds} d{\xi}^n_t\bigg]. 
\end{align}

Notice now that due to \eqref{controlsbar} we have for either $(\vartheta^n, \overline{\vartheta}^n) = (\xi^n, \overline{\xi}^n)$ or $(\vartheta^n, \overline{\vartheta}^n) = (\eta^n, \overline{\eta}^n)$ that
\begin{align}
\label{ui3}
\E\bigg[\int_0^{\infty} e^{-\int_0^t \rho_{Y_s^i}ds} X^{1,i,0}_t \, \overline{\vartheta}^n_t\, dt\bigg] 
& = \E\bigg[\int_0^{\infty} e^{-\int_0^t \rho_{Y_s^i}ds} X^{1,i,0}_t \Big(\int_{0}^t\frac{d\vartheta^n_u}{X^{1,i,0}_u}\Big) dt\bigg]\nonumber \\
& = \E\bigg[\int_0^{\infty} \frac{1}{X^{1,i,0}_u}\E\bigg[\int_u^{\infty}e^{-\int_0^t \rho_{Y_s^i}ds} X^{1,i,0}_t\,dt\Big|\mathcal{F}_u\bigg]\,d\vartheta^n_u\bigg],
\end{align}
where Tonelli's theorem and Theorem 57 in Chapter VI of \cite{DM} imply the last equality.

We now want to find a lower bound for 
$\E[\int_u^{\infty} e^{-\int_0^t \rho_{Y_s^i}ds} X^{1,i,0}_t \,dt |\mathcal{F}_u]/{X^{1,i,0}_u}$. 
To accomplish that we notice that \eqref{solGBM}, the fact that $\lambda_{Y_t} \geq \lambda_N$ and $\rho_{Y_t} \leq \ol\rho$, $\P$-a.s.\ for all $t\geq 0$, and a change of variable of integration give
\begin{align}
\label{estimate1}
\frac{1}{X^{1,i,0}_u}\E\bigg[\int_u^{\infty} e^{-\int_0^t \rho_{Y_s^i}ds} X^{1,i,0}_t\,dt\Big|\mathcal{F}_u\bigg] 
& \geq e^{-\int_0^u \rho_{Y_s^i}ds} \,
\E\bigg[\int_0^{\infty} e^{-(\ol\rho - r + g - \lambda_N + \frac{1}{2}\sigma^2)t} e^{\sigma(W_{t+u}-W_u)} \,dt\Big|\mathcal{F}_u\bigg]\nonumber \\
& = e^{-\int_0^u \rho_{Y_s^i}ds} \int_0^{\infty} e^{-(\ol\rho - r + g - \lambda_N)t} dt 
= e^{-\int_0^u \rho_{Y_s^i}ds} \beta_2, 
\end{align}
where we have set $\beta_2:=(\ol\rho-r+g-\lambda_N)^{-1}<\infty$ by Assumption \ref{ass:rho}. In \eqref{estimate1} the independence of Brownian increments, the stationarity of their distribution, and the formula for the Laplace transform of a Gaussian random variable have been employed in the penultimate step.
Analogously, but using now that $\lambda_{Y_t} \leq \lambda_1$ and $\rho_{Y_t} \geq \ul\rho$, $\P$-a.s.\ for all $t\geq 0$, we find
\begin{align}
\label{estimate2}
\frac{1}{X^{1,i,0}_u}\E\bigg[\int_u^{\infty} e^{-\int_0^t \rho_{Y_s^i}ds} X^{1,i,0}_t\,dt\Big|\mathcal{F}_u\bigg] \leq e^{-\int_0^u \rho_{Y_s^i}ds} \int_0^{\infty} e^{-(\ul\rho - r + g - \lambda_1)t} dt 
= e^{-\int_0^u \rho_{Y_s^i}ds} \beta_1,
\end{align}
with $\beta_1:=(\ul\rho-r+g-\lambda_1)^{-1}<\infty$ by Assumption \ref{ass:rho}.

Recalling \eqref{ui3} and using \eqref{estimate1} and \eqref{estimate2} we then find from \eqref{ui2} that
\begin{align}
\label{ui4}
&1 + {c_1} x + \frac{\kappa_{\varepsilon}}{\ul\rho} \geq 
\big(\varepsilon \beta_2 - {c_2}\big)\E_{(x,i)}\bigg[\int_0^{\infty} e^{-\int_0^t \rho_{Y_s}ds} d{\xi}^n_t \bigg] + \big(c_1-\varepsilon \beta_1 \big)\E_{(x,i)}\bigg[\int_0^{\infty}  e^{-\int_0^t \rho_{Y_s}ds} d{\eta}^n_t \bigg].
\end{align}
The previous estimate holds for any $\varepsilon>0$. Hence setting $\Theta_{\varepsilon}(x):=1 + {c_1} x + \frac{\kappa_{\varepsilon}}{\ul\rho}$, we can take $\varepsilon={c_2}/\beta_2$ in \eqref{ui4} and obtain
$$\beta_2 \,\Theta_{{c_2}/{\beta_2}}(x) \geq \big({c_1}\beta_2 - {c_2}\beta_1)\E_{(x,i)}\bigg[\int_0^{\infty} e^{-\int_0^t \rho_{Y_s}ds} d{\eta}^n_t \bigg].$$
On the other hand, by taking $\varepsilon={c_1}/\beta_1$ in \eqref{ui4} we have
$$\beta_1\Theta_{{c_1}/{\beta_1}}(x) \geq \big({c_1}\beta_2 - {c_2}\beta_1)\E_{(x,i)}\bigg[\int_0^{\infty}  e^{-\int_0^t \rho_{Y_s}ds} d{\xi}^n_t \bigg].$$
Noticing that ${c_1}\beta_2 - {c_2}\beta_1>0$ by Assumption \ref{ass:c12}, the last two inequalities then give
$$\E_{(x,i)}\bigg[\int_0^{\infty}  e^{-\int_0^t \rho_{Y_s}ds} \big(d{\xi}^n_t + d{\eta}^n_t\big)\bigg]
\leq \frac{\beta_2\Theta_{{c_2}/{\beta_2}}(x) + \beta_1\Theta_{{c_1}/{\beta_1}}(x)}{{c_1}\beta_2 - {c_2}\beta_1},$$
which clearly implies \eqref{eq:ui} since the right-hand side of the latter is independent of $n$.
\end{proof}

In view of Lemma \ref{lem:ui}, we can now prove the main result of this section.

\begin{theorem}
\label{thm:existence}
Let $(x,i)\in \mathcal{O}$ be given and fixed. There exists a unique (up to undistinguishability) optimal debt ratio management policy $\varphi^{\star}=\xi^{\star}-\eta^{\star}$ for the problem \eqref{eq:valueOC}.
\end{theorem}
\begin{proof}
Uniqueness (up to undistinguishability) of the optimal debt management policy is due, as usual, to the strict convexity of the cost functional and to the affine structure of the controlled state variable with respect to the control. Therefore, in the following we only prove existence of an optimal control. 

Let $(x,i)\in \mathcal{O}$ be given and fixed, and let $(\varphi^n)_{n\in \mathbb{N}}:=(\xi^n, \eta^n)_{n\in \mathbb{N}}$ be a minimising sequence for problem \eqref{eq:valueOC}. By \eqref{eq:ui} in Lemma \ref{lem:ui} we deduce that 
$$\sup_{n \in \mathbb{N}}\E_{(x,i)}\bigg[\int_0^{\infty} e^{-\int_0^t \rho_{Y_s}ds} \big(\eta^n_t + \xi^n_t\big) dt\bigg] < \infty;$$
that is, $(\varphi^n)_{n\in \mathbb{N}}$ is bounded in $L^1(\Omega\times\mathbb{R}_+, \mu)$, where $\mu := \P(d\omega)\otimes e^{-\int_0^t \rho_{Y_s}ds} dt$. 
Koml\'os' theorem \cite{Komlos} thus implies that there exists a subsequence (still denoted by $(\varphi^n)_{n\in \mathbb{N}}$ for simplicity of notation) and a pair of measurable processes $\xi^{\star}$ and $\eta^{\star}$ such that the Ces\`aro sequences
\begin{equation*}
\widetilde{\xi}^n:=\frac{1}{n}\sum_{j=1}^n \xi^j \to \xi^{\star}
\qquad \mbox{and} \qquad \widetilde{\eta}^n:=\frac{1}{n}\sum_{j=1}^n \eta^j \to \eta^{\star} ,
\qquad \mu-\text{a.e.}
\end{equation*}
Hence, setting $\widetilde{\varphi}^n:=\widetilde{\xi}^n-\widetilde{\eta}^n$ and $\varphi^{\star}:=\xi^{\star}-\eta^{\star}$, we get $\widetilde{\varphi}^n \to \varphi^{\star}$, $\mu$-a.e. Arguing as in Lemmata 4.5-4.7 of \cite{KS84} (notice indeed that our a.e.\ convergence implies the weak convergence employed in that paper) one can show that $\xi^{\star}$ and $\eta^{\star}$ admit modifications -- that we still denote by $\xi^{\star}$ and $\eta^{\star}$ -- that are nondecreasing, left-continuous and $\mathbb{F}$-adapted; that is, $\varphi^{\star} \in \mathcal{V}$, and $X^{x,i,\varphi^{\star}}_t\geq 0$, $\P\otimes dt$-a.e.

Moreover, it follows from Portmanteau theorem (see, e.g., Theorem 2.1 in \cite{Billingsley}) that $\P$-a.s.
$$\lim_{n \uparrow \infty}\int_{0}^{\infty} f_u\, {d\widetilde{\xi}^n_u} = \int_{0}^{\infty} f_u\,{d{\xi}^{\star}_u} \quad \mbox{and} \quad \lim_{n \uparrow \infty}\int_{0}^{\infty} f_u\, {d\widetilde{\eta}^n_u} = \int_{0}^{\infty} f_u\,{d{\eta}^{\star}_u},$$
for any bounded function $f:\mathbb{R}_+ \to \mathbb{R}_+$ that is continuous $d{\xi}^{\star}$-a.e.\ (resp., $d{\eta}^{\star}$-a.e.)\ on $\mathbb{R}_+$.
The latter convergence in particular yields
\begin{equation}
\label{Port1}
\lim_{n \uparrow \infty}\int_{0}^{\infty} e^{-\int_0^u \rho_{Y_s}ds} {d\widetilde{\xi}^n_u} = \int_{0}^{\infty} e^{-\int_0^u \rho_{Y_s}ds} {d{\xi}^{\star}_u}
\quad \mbox{and} \quad 
\lim_{n \uparrow \infty}\int_{0}^{\infty} e^{-\int_0^u \rho_{Y_s}ds} {d\widetilde{\eta}^n_u} = \int_{0}^{\infty} e^{-\int_0^u \rho_{Y_s}ds} {d{\eta}^{\star}_u},
\end{equation}
which by Fatou's lemma and \eqref{eq:ui} gives $\E_{(x,i)}[\int_0^{\infty} e^{-\int_0^t \rho_{Y_s}ds} d\eta^{\star}_t + \int_0^{\infty} e^{-\int_0^t \rho_{Y_s}ds} d\xi^{\star}_t] < \infty$, and therefore $\varphi^{\star} \in \mathcal{A}$. Furthermore, we have $\P$-a.s.\ for a.e.\ $t\geq0$ that
\begin{equation}
\label{Port2}
\lim_{n \uparrow \infty}\int_{0}^{\infty}\mathds{1}_{[0,t)}(s)\frac{d\widetilde{\xi}^n_s}{X^{1,i,0}_s} = \int_{0}^{\infty}\mathds{1}_{[0,t)}(s)\frac{d{\xi}^{\star}_s}{X^{1,i,0}_s}= \overline{\xi_t^{\star}} 
\end{equation}
\begin{equation}
\label{Port3}
\lim_{n \uparrow \infty}\int_{0}^{\infty}\mathds{1}_{[0,t)}(s)\frac{d\widetilde{\eta}^n_s}{X^{1,i,0}_s}=\int_{0}^{\infty}\mathds{1}_{[0,t)}(s)\frac{d{\eta}^{\star}_s}{X^{1,i,0}_s}=\overline{\eta_t^{\star}} 
\end{equation}
upon recalling \eqref{controlsbar} to have the last two equalities in \eqref{Port2} and \eqref{Port3}.

If we can now apply Fatou's lemma to $\mathcal{J}_{x,i}(\widetilde{\varphi}^{n})$ from \eqref{eq:J} in view of the limits \eqref{Port1}--\eqref{Port3} and the expressions \eqref{solDTGDP} and \eqref{Xbar} of $X^{x,i,\varphi}$, we obtain that
\begin{equation}
\label{existencefinal}
\mathcal{J}_{x,i}(\varphi^{\star}) \leq \liminf_{n\uparrow \infty}\mathcal{J}_{x,i}(\widetilde{\varphi}^{n}) \leq V(x,i),
\end{equation}
where we have used that $(\widetilde{\varphi}^{n})_{n\mathbb{N}}$ is also a minimising sequence due to the convexity of $\mathcal{J}_{x,i}(\,\cdot\,)$ on $\mathcal{V}$. Hence, $\varphi^{\star}$ is optimal.

Therefore, in order to complete the proof of this part, we show in the remaining that Fatou's lemma can be indeed applied. 
Using the change of measure from \eqref{CoM} on the expression of $\mathcal{J}_{x,i}(\cdot)$ involved in \eqref{reprV}, we can write (see \eqref{rhohat} as well)
\begin{align*}
&\mathcal{J}_{x,i}(\varphi^n) = 
\EE\bigg[\int_0^{\infty} e^{-\int_0^t \rho_{Y_s^i}ds} \frac{1}{M_t} h\big(X^{1,i,0}_t\big[x + \overline{\xi}^n_t  - \overline{\eta}^n_t\big], Y_t^i\big) dt\bigg] 
+ \EE_{(x,i)}\bigg[\int_0^{\infty}  e^{-\rrho_t} 
\big( c_1 d\overline{\eta}^n_t - c_2 d\overline{\xi}^n_t\big)\bigg] \nonumber \\
&= \EE\bigg[\int_0^{\infty}  \hspace{-.15cm}e^{-\int_0^t \rho_{Y_s^i}ds} \frac{1}{M_t} h\big(X^{1,i,0}_t\big[x + \overline{\xi}^n_t  - \overline{\eta}^n_t\big], Y_t^i\big) dt\bigg]  
+ \EE_{(x,i)}\bigg[\int_0^{\infty} \hspace{-.15cm} e^{-\rrho_t} 
\big(\rho_{Y_t} - \lambda_{Y_t} - r + g)\big)\big(c_1 \overline{\eta}^n_t - c_2 \overline{\xi}^n_t\big)dt\bigg], \nonumber
\end{align*}
where an integration by parts for the integrals with respect to $d\overline{\eta}^n_t$ and $d\overline{\xi}^n_t$ and \eqref{eq:ui} have been used to obtain the last equality.
Thus, by defining the random variable
\begin{align*}
\Phi_n := 
\int_0^{\infty}  \hspace{-.15cm}e^{-\int_0^t \rho_{Y_s^i}ds} \frac{1}{M_t} h\big(X^{1,i,0}_t\big[x + \overline{\xi}^n_t  - \overline{\eta}^n_t\big], Y_t^i\big) dt 
+ \int_0^{\infty} \hspace{-.15cm} e^{-\rrho_t} 
\big(\rho_{Y_t^i} - \lambda_{Y_t^i} - r + g)\big)\big(c_1 \overline{\eta}^n_t - c_2 \overline{\xi}^n_t\big)dt
\end{align*}
we will prove that Fatou's lemma can be applied in \eqref{existencefinal}, if we find an integrable random variable $\Lambda$, independent of $n$, such that $\Phi_n\geq \Lambda, \PP$-a.s.
To this end, using that $\lambda_N \leq \lambda_{Y^i_t}\leq \lambda_1$ and $\ul\rho \leq \rho_{Y^i_t}\leq \ol\rho$, $\PP$-a.s., and that for any $\varepsilon > 0$ there exists $\kappa_{\varepsilon}>0$ such that $h(x,i)\geq \varepsilon x - \kappa_{\varepsilon}$ for any $(x,i) \in \mathbb{R}_+ \times \M$ (cf.\ Assumption \ref{ass:h}-(iv)) together with \eqref{XM}, we can write $\PP$-a.s.\ that
\begin{align*}
\Phi_n  \geq 
&-\kappa_{\varepsilon}\int_0^{\infty} e^{-\ul\rho t} \frac{1}{M_t} dt +\frac{\varepsilon x}{\ol\rho - r + g -\lambda_{N}} + 
\big(\varepsilon- c_2(\ol\rho - r + g -\lambda_{N})\big) \int_0^{\infty} e^{-\rrho_t} 
\,\overline{\xi}^n_t dt \nonumber \\
&+ \big(c_1(\ul\rho - r + g -\lambda_{1}) - \varepsilon\big) \int_0^{\infty} e^{-\rrho_t} 
\,\overline{\eta}^n_t dt. 
\end{align*}
Therefore, by taking $\varepsilon = c_1(\ul\rho - r + g - \lambda_{1})$ in the above expression, and using Assumption \ref{ass:c12}, we obtain
\begin{align}
\label{Fatou4}
\Phi_n  >
&- \kappa_{\varepsilon} \int_0^{\infty} e^{-\ul\rho t} \frac{1}{M_t} dt + \frac{\ul\rho - r + g - \lambda_{1}}{\ol\rho - r + g -\lambda_{N}} \, c_1 x =: \Lambda. \nonumber 
\end{align}
The fact that $\Lambda$ is clearly an integrable random variable, independent of $n$, completes the proof.
\end{proof}

The previous theorem ensures existence and uniqueness of an optimal debt ratio management policy, but it does not directly provide its structure. To determine the form of the optimal debt ratio management policy, we now exploit the result of Theorem \ref{thm:existence} and we relate the optimal debt management problem to a two person zero-sum game of optimal stopping with regime switching. 

We now provide a probabilistic representation of $V_x$.
\begin{proposition}
\label{prop:Vx}
For any $(x,i) \in \mathcal{O}$ set
\begin{equation}
\label{stfunct}
\Psi_{x,i}(\tau,\theta):=\E\bigg[\int_{0}^{\tau \wedge \theta}
\hspace{-.18cm}e^{-\int_0^t \rho_{Y_s^i}ds} X^{1,i,0}_t h_x\big(x X^{1,i,0}_t, Y_t^i\big)dt 
+ c_2 e^{-\int_0^\tau \rho_{Y_s^i}ds} X^{1,i,0}_{\tau}\mathds{1}_{\{\tau < \theta\}} + c_1 e^{-\int_0^\theta \rho_{Y_s^i}ds} X^{1,i,0}_{\theta}\mathds{1}_{\{\theta < \tau\}}\bigg],
\end{equation}
for a couple of $\mathbb{F}$-stopping times ($\tau, \theta)$.
Then,  
\begin{equation}
\label{Vx}
V_x(x,i) = v(x,i), \qquad (x,i) \in \mathcal{O},
\end{equation}
where $v$ is the value function of the zero-sum Dynkin game with regime switching
\begin{equation}
\label{valueOS}
v(x,i):=\sup_{\tau \geq 0}\inf_{\theta \geq 0}\Psi_{x,i}(\tau,\theta) = \inf_{\theta \geq 0}\sup_{\tau \geq 0}\Psi_{x,i}(\tau,\theta), \quad (x,i) \in \mathcal{O}.
\end{equation}
\end{proposition}
\begin{proof}
For any $(x,i)\in \mathcal{O}$, $t\geq 0$ and $\omega \in \Omega$, recall \eqref{solGBM} and set
\begin{align}
\label{set}
&H(\omega,t,x):= e^{-\int_0^t \rho_{Y_s^i(\omega)}ds} 
h\big(x \cdot X^{1,i,0}_t(\omega),  Y_t^i(\omega)\big), \nonumber\\ 
&\nu_t(\omega):= c_1 e^{-\int_0^t \rho_{Y_s^i(\omega)}ds} 
X^{1,i,0}_t(\omega), 
\quad 
\gamma_t(\omega):= - c_2 e^{-\int_0^t \rho_{Y_s^i(\omega)}ds} X^{1,i,0}_t(\omega).
\end{align}
Due to Assumptions \ref{ass:rho} and \ref{ass:h}, and standard estimates, it is easy to check that
\begin{equation*}
{(i)}\,\,\,\E\Big[\sup_{t\geq 0}|\gamma_t| + \sup_{t\geq 0}|\nu_t|\Big] < \infty, \qquad {(ii)}\,\,\,\E\bigg[\int_0^{\infty}|H_x(\omega,t,x)| dt\bigg] < \infty.
\end{equation*}
We thus have that the integrability conditions required in equation (2.4) of \cite{KaratzasWang} are satisfied, and we can therefore apply Theorems 3.1 and 3.2 of \cite{KaratzasWang} together with our Theorem \ref{thm:existence} in order to conclude. In fact, going through the proofs of Theorems 3.1 and 3.2 of \cite{KaratzasWang}, one should notice that the required nonnegativity of the process $\gamma$ is not necessary. The arguments of those proofs still work in the case (as in the present paper) in which $\gamma$ is negative (cf.\ \eqref{set}) and $\E\big[\sup_{t\geq 0}|\gamma_t|\big] < \infty$.
Moreover, it is important to remark that in \cite{KaratzasWang} the set of admissible controls does not require that the controlled process remains positive. However, the proof of Theorem 3.2 therein is based on the construction of suitable perturbations of the optimal control and one may easily verify that such perturbations of the optimal control preserve positivity of the process provided that the optimal control does.
\end{proof}

This game might be interpreted as a game played between the two components of the government; namely, player 1 (inf--player choosing $\theta$) represents the will to adopt a restrictive debt policy and player 2  (sup--player choosing $\tau$) represents the desire to increase spending.

\section{The Associated Optimal Stopping Game}
\label{sec:OSgame}

In this section we will study the Dynkin game with regime switching with value \eqref{valueOS}. In particular, we will characterise the saddle point of the game as a couple of hitting times of two regime dependent boundaries, and we will prove global $C^1$-regularity of $v(\cdot, i)$ for any $i\in \mathcal{M}$. This study will be crucial for the identification of the optimal control of problem \eqref{eq:valueOC}, completely characterising the optimal debt management policy of the government, developed in Section \ref{sec:OCrule}.

For the subsequent analysis, we define the process
\begin{equation}
\label{rhohat}
\rrho_t:=\int_0^t (\rho_{Y_s} - \lambda_{Y_s}) ds - (r - g)t , \quad t\geq 0,
\end{equation}
and let $\PP$ be the measure on $(\Omega, \mathcal{F})$ such that 
\begin{equation}
\label{CoM}
\frac{d\PP}{d\P}\Big|_{\mathcal{F}_t} = M_t \, , \quad \text{for } t \geq 0, \qquad \text{with} \qquad M_t:=\exp\Big\{-\frac{1}{2}\sigma^2 t + \sigma W_t \Big\}, 
\end{equation}
and denote by $\EE_{(x,i)}$ the expectation under $\PP$ conditioned on $X_0=x$ and $Y_0=i$, for $(x,i) \in \mathcal{O}$. 
Notice that for any $t\geq 0$, we can rewrite $X^{x,i,0}$ from \eqref{solGBM} as
\begin{equation}
\label{XM}
X^{x,i,0}_t=x\exp\Big\{(r-g)t + \int_0^t \lambda_{Y^i_s}\, ds \Big\} \, M_t 
= x \exp\Big\{ \int_{0}^{t} \rho_{Y_s^i} ds - \rrho_t \Big\} \, M_t.
\end{equation}
In view of the change of measure in \eqref{CoM}, we have by Girsanov's theorem that $\WW_t:= W_t - \sigma t$ is a standard $\mathbb{F}$-Brownian motion under $\PP$, and we introduce the process (cf.\ \eqref{solGBM})
\begin{equation}
\label{newGBM}
\displaystyle {\X}^{x,i,0}_t = x e^{(r - g + \frac{1}{2}\sigma^2)t + \int_0^t \lambda_{Y^i_s}ds + \sigma \WW_t}, \qquad t\geq 0.
\end{equation}
Moreover, we can rewrite (cf.\ \eqref{valueOS})
\begin{equation}
\label{valueOS2}
v(x,i)=\sup_{\tau \geq 0}\inf_{\theta \geq 0}\PPsi_{x,i}(\tau,\theta) = \inf_{\theta \geq 0}\sup_{\tau \geq 0}\PPsi_{x,i}(\tau,\theta), \quad (x,i) \in \mathcal{O},
\end{equation}
where for every couple of $\mathbb{F}$-stopping times $(\tau,\theta)$ we have set
\begin{equation}
\label{stfunct-bis}
\PPsi_{x,i}(\tau,\theta):={\EE}_{(x,i)}\bigg[\int_{0}^{\tau \wedge \theta}e^{-\rrho_t}h_x\big({\X}^0_t, Y_t \big)dt + c_2 e^{-\rrho_\tau}\mathds{1}_{\{\tau < \theta\}} + c_1 e^{-\rrho_\theta}\mathds{1}_{\{\theta < \tau\}}\bigg].
\end{equation}
with $\rrho_\cdot$ given by \eqref{rhohat}.

It is easy to see that since $\ul\rho > r - g + \lambda_1$ by Assumption \ref{ass:rho}, then $\lim_{t \uparrow \infty}e^{-\rrho_t}=0$, $\PP$-a.s.. Therefore, in the rest of this section, for any $\mathbb{F}$-stopping time $\zeta$ we will adopt the convention
$$ e^{-\rrho_\zeta} := 0 \quad \mbox{on} \quad \{\zeta=+\infty\}.$$

From \eqref{valueOS2}--\eqref{stfunct-bis} it is readily seen that $c_2 \leq v(x,i) \leq c_1$. Using the general theory of optimal stopping for Markov processes (see, e.g., Chapter 2 of \cite{PeskirShir}) define the \emph{continuation region}
$$\mathcal{C}:=\{(x,i)\in \mathcal{O}:\, c_2 < v(x,i) < c_1\},$$
and the \emph{stopping regions}
$$\mathcal{S}_1:=\{(x,i)\in \mathcal{O}:\,v(x,i) \geq c_1\}, \qquad \mbox{and} \qquad \mathcal{S}_2:=\{(x,i)\in \mathcal{O}:\,v(x,i) \leq c_2\}.$$

Here $\mathcal{C}$ is the region in which no player has an incentive to stop the evolution of the process $(\X^0,Y)$, whereas $S_j$, $j=1,2$, is the region in which it is optimal for player $j$ to stop.

Since $x \mapsto {\X}^{x,i,0}_{t}$ is $\PP$-a.s.\ increasing (cf.\ \eqref{newGBM}), it follows from \eqref{valueOS2} that $x \mapsto v(x,i)$ is increasing for any $i\in \mathcal{M}$ due to the convexity of $h(\cdot,i)$.
Hence we can introduce the \emph{free boundaries}
\begin{equation}
\label{bds}
a(i):=\inf\{x \geq 0: v(x,i) > c_2\} \quad \text{and} \quad b(i):=\sup\{x \geq 0: v(x,i) < c_1\},
\end{equation}
(with the usual convention $\sup\emptyset=0$ and $\inf \emptyset = + \infty$), and we have that $\mathcal{O}= \R_+ \times \mathcal{M}$ is split into continuation and stopping regions completely determined by $a$ and $b$; that is,
$$\mathcal{C}=\{(x,i)\in \mathcal{O}:\, a(i) < x < b(i)\},\quad \mathcal{S}_1=\{(x,i)\in \mathcal{O}:\, x \geq b(i)\}, \quad\mathcal{S}_2=\{(x,i)\in \mathcal{O}:\, x \leq a(i)\}.$$

The Markov process $({\X}^0,Y)$ has c\`adl\`ag paths and it is of Feller type by \cite{Zhou-Yin} (see Lemma 3.6 and Theorem 3.10 therein). Hence its paths are right-continuous and quasi-left-continuous (i.e.\ left-continuous over predictable stopping times), and by Theorem 2.1 of \cite{EkstromPeskir} we know that $\PP_{(x,i)}$-a.s., for any $(x,i) \in \mathcal{O}$, the two stopping times
\begin{equation}
\label{Nashst}
\theta^{\star}:=\inf\{t\geq 0:\, ({\X}^0_t,Y_t) \in \mathcal{S}_1\}\quad \mbox{and} \quad \tau^{\star}:=\inf\{t\geq 0:\, ({\X}^0_t,Y_t) \in \mathcal{S}_2\},
\end{equation}
form a saddle point for the game \eqref{valueOS2} (here the usual convention $\inf \emptyset = + \infty$ applies). Moreover, by easily adapting the results of Theorem 2.1 in \cite{PeskirNash} to our case with running cost $h_x$, we also have the following probabilistic characterisation of $v$. Such a result is usually referred to as the \emph{semi-harmonic characterisation of $v$}. 

\begin{proposition}
\label{prop:semih}
For any $(x,i) \in \mathcal{O}$, we have under $\PP_{(x,i)}$ that
\begin{itemize}
\item[(i)]
$\big(\int_0^{t\wedge \tau^{\star}}e^{-\rrho_s}h_x\big({\X}^0_s, , Y_s \big) ds + e^{-\rrho_{t\wedge \tau^{\star}}}v({\X}^0_{t\wedge \tau^{\star}},Y_{t\wedge \tau^{\star}})\big)_{t\geq0}$  
is a right-continuous $\mathbb{F}$-submartingale;
\item[(ii)]
$\big(\int_0^{t\wedge \theta^{\star}}e^{-\rrho_s}h_x\big({\X}^0_s, , Y_s\big) ds + e^{-\rrho_{t\wedge \theta^{\star}}}v({\X}^0_{t\wedge \theta^{\star}},Y_{t\wedge \theta^{\star}})\big)_{t\geq0}$
is a right-continuous $\mathbb{F}$-supermartingale;
\item[(iii)]
$\big(\int_0^{t\wedge \theta^{\star}\wedge \tau^{\star}}e^{-\rrho_s}h_x\big({\X}^0_s, Y_s\big) ds + e^{-\rrho_{t\wedge \theta^{\star}\wedge \tau^{\star}}}v({\X}^0_{t\wedge \theta^{\star}\wedge \tau^{\star}},Y_{t\wedge \theta^{\star}\wedge \tau^{\star}})\big)_{t\geq0}$
is a right-continuous $\mathbb{F}$-martingale;
\end{itemize}
\end{proposition}

The following proposition rules out the possibility that the stopping regions are empty, thus the boundaries $a(i)$ and $b(i)$ from \eqref{bds} exist and are finite under any regime $i \in \mathcal{M}$, and the optimal stopping times in \eqref{Nashst}, forming the Nash-equilibrium, are well-defined.

\begin{proposition}
\label{notemptyS1}
The following hold true:
\begin{itemize}
\item[(i)] $\mathcal{S}_1 \neq \emptyset$ and $\mathcal{S}_2 \neq \emptyset$; 
\item[(ii)] there exist constants $0<a_1<b_1<+\infty$ and $0<a_N<b_N<+\infty$, with $a_1 \leq a_N$ and $b_1 \leq b_N$, such that for all $i\in \mathcal{M}$ we have $a_1 \leq a(i) \leq a_N$ and $b_1 \leq b(i) \leq b_N$. 
\end{itemize}
\end{proposition}
\begin{proof}
We prove the two claims separately.
\vspace{0.15cm}

\emph{Proof of (i)}. We argue by contradiction and we suppose that $\mathcal{S}_1 = \emptyset$. This implies that $\theta^{\star}=+\infty$ $\PP_{(x,i)}$-a.s.\ for any $(x,i) \in \mathcal{O}$ and therefore
\begin{align*}
c_1 > v(x,i) &= \displaystyle \sup_{\tau \geq 0}{\EE}_{(x,i)}\bigg[\int_{0}^{\tau}e^{-\rrho_t}h_x\big({\X}^0_t, Y_t\big)dt + c_2 e^{-\rrho_\tau}\bigg] 
\geq \displaystyle {\EE}\bigg[\int_{0}^{T}e^{-\rrho_t}h_x\big(x \cdot \X^{1,i,0}_t, Y_t\big)dt + c_2 e^{-\rrho_T}\bigg], 
\end{align*}
for $T>0$ deterministic. By letting $x \uparrow \infty$, and recalling that $h_x(x,i) \uparrow \infty$ by Assumption \ref{ass:h}, we obtain by the monotone convergence theorem that the last expected value diverges to $+\infty$, thus leading to a contradiction. 

Given that we allow the process ${\X}^{0}$ to start from $x=0$ at time $t=0$, in which case ${\X}^{0}_t\equiv 0$ for all $t\geq0$, $\PP$-a.s., and $h_x(0,i)=0$ by Assumption \ref{ass:h}-(ii), we clearly have that $v(0,i)=c_2$ for any $i \in \mathcal{M}$.  That is, the minimiser chooses $\theta^\star=+\infty$ and the maximiser $\tau^\star=0$ in \eqref{Nashst}. Thus, $(0,i) \in \mathcal{S}_2$ for any $i\in \mathcal{M}$, which yields that the stopping set $\mathcal{S}_2 \neq \emptyset$.
\vspace{0.15cm}

\emph{Proof of (ii)}. Since $\lambda_{Y_t} \in [\lambda_N,\lambda_1]$ and $\rho_{Y_t} \in [\ul\rho, \ol\rho]$, $\P$-a.s.\ for all $t\geq0$ (see Assumption \ref{lY}), it is straightforward to see that $v_N(x) \leq v(x,i) \leq v_1(x)$, for all $x\geq0$, $i\in \mathcal{M}$. The bounds $v_k(x)$, for $k\in\{1,N\}$, are defined by
\begin{align}
\label{upperv}
&v_k(x):=\sup_{\tau\geq0}\inf_{\theta\geq0}\Xi^{\footnotesize{(k)}}_x(\tau,\theta)= \inf_{\theta\geq0}\sup_{\tau\geq0}\Xi^{\footnotesize{(k)}}_x(\tau,\theta), 
\end{align}
with
\begin{align*}
&\Xi^{\footnotesize{(1)}}_x(\tau,\theta):={\EE}\bigg[\int_{0}^{\tau \wedge \theta}\hspace{-0.25cm}e^{-(\ul\rho - \lambda_1 -r + g)t} \max_i h_x\big(Z^{(1),x}_t, i\big)dt 
+ c_2 e^{-(\ul\rho - \lambda_1 -r + g)\tau}\mathds{1}_{\{\tau < \theta\}} + c_1 e^{-(\ul\rho - \lambda_1 -r + g)\theta}\mathds{1}_{\{\theta < \tau\}}\bigg] \\
&\Xi^{\footnotesize{(N)}}_x(\tau,\theta):=\EE\bigg[\int_{0}^{\tau \wedge \theta}\hspace{-0.25cm}e^{-(\ol\rho - \lambda_N -r + g)t} \min_i h_x\big(Z^{(N),x}_t, i\big)dt 
+ c_2 e^{-(\ol\rho - \lambda_N -r + g)\tau}\mathds{1}_{\{\tau < \theta\}} + c_1 e^{-(\ol\rho - \lambda_N -r + g)\theta}\mathds{1}_{\{\theta < \tau\}}\bigg]
\end{align*}
for $Z^{\footnotesize{(k)},x}_t= x \exp \{(r-g + \frac{1}{2}\sigma^2 + \lambda_k)t + \sigma\widetilde{W}_t \}$, for all $t\geq 0$. By defining the free boundaries of the one-dimensional (without regime switching) zero-sum optimal stopping games \eqref{upperv}, for any $k\in\{1,N\}$, by
$$a_k:=\inf\{x \geq 0: v_k(x) > c_2\} \quad \text{and} \quad b_k:=\sup\{x \geq 0: v_k(x) < c_1\},$$
we apply standard means to prove that these constants exist and are such that $0<a_k<b_k<+\infty$ (compare also with our analysis of Section \ref{sec:comparison}). Moreover, $a_1 \leq a_N$ and $b_1 \leq b_N$. Thus, using the fact that $v_N(x) \leq v(x,i) \leq v_1(x)$, it is easy to see that $a_1 \leq a(i) \leq a_N$ and $b_1 \leq b(i) \leq b_N$, which completes the proof.
\end{proof}

For any $i \in \mathcal{M}$ introduce the $i$-sections for $\mathcal{C}$, $\mathcal{S}_1$ and $\mathcal{S}_1$ as
\begin{equation*}
\mathcal{C}^i:=\{x \geq 0:\, (x,i) \in \mathcal{C}\} \quad \mbox{and} \quad \mathcal{S}_j^i:=\{x \geq 0:\, (x,i) \in \mathcal{S}_j\}, \quad \mbox{for} \quad j=1,2.
\end{equation*}
The next result proves regularity of $x \mapsto v(x,i)$ for any $i \in \mathcal{M}$.
\begin{theorem}
\label{thm:reg-v}
For any $i \in \mathcal{M}$, 
\begin{itemize}
\item[(i)] 
$v(\cdot,i) \in C^{2}\big((\mathcal{C}^i\cup \mathcal{S}_1^i \cup \mathcal{S}_2^i) \setminus \{a(i),b(i)\}$\big); 
\item[(ii)] 
$v(\cdot,i) \in C^1(\mathbb{R}_+)$. 
\end{itemize}
\end{theorem}
\begin{proof}
We prove the two parts separately. 
\vspace{0.15cm}

\emph{Proof of (i).} Clearly, for any $i\in \mathcal{M}$, $v(\cdot,i) \in C^{2}(\mathcal{S}_1^i \cup \mathcal{S}_2^i) \setminus \{a(i),b(i)\}$ since $v \equiv c_1$ in $\mathcal{S}_1^i \setminus \{b(i)\}$ and $v \equiv c_2$ in $\mathcal{S}_2^i\setminus \{a(i)\}$. Thus, what remains to be proved is that $v(\cdot,i) \in C^{2}(\mathcal{C}^i)$, which is presented below. 

Let $i \in \mathcal{M}$ be given and fixed, and let $\alpha < \beta$ such that $[\alpha,\beta] \subset \mathcal{C}^i=\{x \geq 0:\, a(i)<x<b(i)\}$.
Then, setting $f(x,i):= h_x(x,i) + \sum_{j \neq i}q_{ij} v(x,j)$, for any $x\in (\alpha,\beta)$, consider a function $w(\cdot,i):\R_+ \mapsto \R$ that solves the ordinary differential equation 
\begin{align}
\label{ODE} 
\frac{1}{2}\sigma^2 x^2 w_{xx}(x,i) + (r - g + \lambda_i + \sigma^2)x w_x(x,i) - \big(\rho_i -\lambda_i - r + g - q_{ii}\big)w(x,i)  
= -f(x,i),
\end{align}
with boundary conditions $w(\alpha,i)=v(\alpha,i)$ and $w(\beta,i)=v(\beta,i)$.
Since $x \geq \alpha > a(i)>0$, the differential operator in \eqref{ODE} is uniformly elliptic and the solution $w$ of the above Dirichlet problem is unique and is such that $w(\cdot,i)\in C^2((\alpha,\beta))$. 
Then, using this function $w$ and recalling that $i \in \mathcal{M}$ is given and fixed, define the function $\overline{w}: (\alpha,\beta) \times \mathcal{M} \mapsto \R$ as follows:
\begin{equation}
\label{extensionw} 
\overline{w}(x,j):=
\begin{cases}
\displaystyle w(x,i)\,\,\qquad\,\,\, &\mbox{if $j=i$}\\[+1pt]
\displaystyle v(x,j)\,\,\quad\,  &\mbox{if $j \neq i$}.
\end{cases}
\end{equation}
In addition, for $x \in (\alpha,\beta)$, let $\tau_{\alpha,\beta}:=\inf\{t\geq 0: {\X}^{x,i,0}_t \notin (\alpha,\beta)\}$, $\tau_1:=\inf\{t\geq 0: Y^i_t \neq i\}$, and set $\zeta:=\tau_{\alpha,\beta} \wedge \tau_1$. Given that $Y_t=i$ for all $t < \zeta$, Dynkin's formula yields that
\begin{equation}
\label{repr:solODE}
\overline{w}(x,i) = w(x,i) = \EE_{(x,i)}\bigg[e^{-\rrho_\zeta}v({\X}^0_{\zeta}, Y_{\zeta}) + \int_0^{\zeta}e^{-\rrho_t}h_x({\X}^0_t, i) dt\bigg],
\end{equation}
due to \eqref{extensionw}, which implies that $\overline{w}({\X}^0_{\zeta}, Y_{\zeta}) = v({\X}^0_{\zeta}, Y_{\zeta})$, and \eqref{ODE}, which implies that
\begin{multline}
\frac{1}{2}\sigma^2 x^2 \overline{w}_{xx}(x,i) + (r - g + \lambda_i + \sigma^2)x \overline{w}_x(x,i) - \big(\rho_i - \lambda_i -r+g\big)\overline{w}(x,i) 
\nonumber\\ 
+ \sum_{j\neq i}q_{ij}\big[\overline{w}(x,j)-\overline{w}(x,i)] +  h_x(x,i) 
\end{multline}
\vspace{-0.65cm}
\begin{multline}
= \frac{1}{2}\sigma^2 x^2 w_{xx}(x,i) + (r - g + \lambda_i + \sigma^2)x w_x(x,i) - \big(\rho_i - \lambda_i -r+g - q_{ii}\big)w(x,i) + f(x,i) = 0 . \nonumber
\end{multline}

However, since $[\alpha, \beta] \subset \mathcal{C}^i$, we have $\zeta \leq \tau^{\star} \wedge \theta^{\star}$, hence it follows from Proposition \ref{prop:semih}-(iii), that the right-hand side of \eqref{repr:solODE} is equal to $v(x,i)$. Therefore, $w \equiv v$ in $(\alpha,\beta) \times \mathcal{M}$ by the arbitrariness of $i$. Also, by the arbitrariness of $(\alpha,\beta)$, we conclude that $w=v$ in $\mathcal{C}$, hence $v(\cdot,i)\in C^2(\mathcal{C}^i)$ for any $i \in \mathcal{M}$.  
\vspace{0.15cm}

\emph{Proof of (ii).} 
We first prove that $v(\cdot,i) \in C^0(\mathbb{R}_+)$ for any $i\in \mathcal{M}$. Since $x \mapsto v(x,i)$ is increasing, we get for any arbitrary $\varepsilon \in (0,1)$ and $(x,i) \in \mathcal{O}$ that 
\begin{align}
\label{eq:vcont}
\displaystyle 0 \leq v(x+\varepsilon,i) - v(x,i) 
&\leq \displaystyle \EE\bigg[\int_{0}^{\infty}e^{-\rrho_t}\big|h_x\big((x+\varepsilon)\cdot \X^{1,i,0}_t, Y_t^i\big)-h_x\big(x \cdot \X^{1,i,0}_t, Y_t^i\big)\big|dt\bigg].
\end{align}
Since 
$|h_x((x+\varepsilon) \cdot \X^{1,i,0}_t, Y_t^i)-h_x(x \cdot \X^{1,i,0}_t, Y_t^i)| \leq 2h_x((x+1)\cdot \X^{1,i,0}_t, Y_t^i)$, $\PP$-a.s.\ 
and 
$\EE[\int_{0}^{\infty}e^{-\rrho_t} h_x\big((x+1)\cdot \X^{1,i,0}_t, Y_t^i \big) dt] < \infty$
due to Assumptions \ref{ass:h}-(iii) and \ref{ass:rho}, we can take limits as $\varepsilon \downarrow 0$ and invoke the dominated convergence theorem in \eqref{eq:vcont} to obtain the claimed continuity of $v(\cdot,i)$ for any $i\in \mathcal{M}$.
\vspace{0.1cm}

In view of the result in part $(i)$ and of the continuity of $v$ proved above, it suffices to show that $v_x(\cdot,i)$ is continuous across the free boundaries $a(i)$ and $b(i)$, for any $i\in \mathcal{M}$. We provide details only for the continuity of $v_x(x,i)$ at $x=a(i)$. Similar arguments apply to show also the continuity of $v_x(x,i)$ at $x=b(i)$.

Take again an arbitrary $(x,i) \in \mathcal{C}$, set $\theta^{\star}:=\theta^{\star}(x,i)=\inf\{t \geq 0: {\X}^{x,i,0}_t \geq b(Y^i_t)\}$ and for a sufficiently small $\varepsilon>0$, set $\tau^{\star}_{\varepsilon}:=\tau^{\star}(x+\varepsilon,i)=\inf\{t \geq 0: \X^{x+\varepsilon,i,0}_t \leq a(Y^i_t)\}$. Then, recalling that $x \mapsto v(x,i)$ is increasing, we can write by Assumption \ref{ass:h}-(iii)
\begin{align*}
0 \leq \displaystyle \frac{v(x+\varepsilon,i)-v(x,i)}{\varepsilon} &\leq \frac{1}{\varepsilon} \, \EE\bigg[\int_0^{\tau^{\star}_{\varepsilon} \wedge \theta^{\star}} e^{-\rrho_t} \Big| h_x\big((x+\varepsilon)\cdot \X^{1,i,0}_t, Y_t^i\big) - h_x\big(x\cdot \X^{1,i,0}_t, Y_t^i\big)\Big| dt\bigg] \nonumber \\
&\leq  K_3 \, \EE\bigg[\int_0^{\tau^{\star}_{\varepsilon} \wedge \theta^{\star}} e^{-\rrho_t} \X^{1,i,0}_t \big[1 + \big(\X^{x+\varepsilon,i,0}_t\big)^{(m-2)^+}\big] dt\bigg].
\end{align*}
Letting $\varepsilon \downarrow 0$, noticing that $\tau^{\star}_{\varepsilon} \rightarrow \tau^{\star}$, $\PP$-a.s., and invoking the dominated convergence theorem thanks to Assumption \ref{ass:rho} yields
\begin{equation*}
0 \leq v_x(x,i) \leq K_3 \, \EE\bigg[\int_0^{\tau^{\star} \wedge \theta^{\star}} e^{-\rrho_t} \X^{1,i,0}_t \big[1 + x^{(m-2)^+}\cdot \big(\X^{1,i,0}_t\big)^{(m-2)^+}\big] dt\bigg].
\end{equation*}
Then by taking limits as $x \downarrow a(i)$ in the latter expression we obtain $v_x(a(i)+,i)=0$. Given that $v(x,i)=c_2$ for all $x \leq a(i)$ we conclude that $v_x(\cdot,i)$ is continuous at $x=a(i)$. 
\end{proof}

\section{The Optimal Debt Management Rule}
\label{sec:OCrule}

Combining Theorem \ref{thm:reg-v} with Proposition \ref{prop:Vx} we immediately have for any $i \in \mathcal{M}$, that $V(\cdot,i) \in C^2(\mathbb{R}_+)$. Hence by the Dynamic Programming Principle (see, e.g., \cite{FlemingSoner}, Chapter VIII.5; see also \cite{BouchardTouzi}, in particular Remarks 3.10 and 3.11, for a proof in a very general setting)
\begin{equation*}
V(x,i) = \inf_{\varphi \in \mathcal{A}}\E_{(x,i)}\bigg[e^{-\int_0^\tau \rho_{Y_s}ds}V(X^{\varphi}_{\tau}, Y_{\tau}) + \int_{0}^{\tau}e^{-\int_0^t \rho_{Y_s}ds} h(X^{\varphi}_t, Y_t) dt + \int_{0}^{\tau} e^{-\int_0^t \rho_{Y_s}ds} \big(c_1d\eta_t - c_2d\xi_t\big)\bigg],
\end{equation*}
for any $\mathbb{F}$-stopping time $\tau$, $V$ identifies with a classical solution to the Hamilton-Jacobi-Bellman (HJB) equation
\begin{equation}
\label{HJB}
\min\big\{\big(\mathcal{G}-\rho_i)V(x,i) + h(x,i), -c_2 + V_x(x,i), c_1 - V_x(x,i)\big\}=0, \qquad (x,i) \in \mathcal{O} . 
\end{equation}
Here $\mathcal{G}$ is the infinitesimal generator of $(X^0,Y)$, which acts on functions $f: \mathcal{O} \to \mathbb{R}$ with $f(\cdot, i) \in C^2(\mathbb{R})$ for any given and fixed $i\in \mathcal{M}$ as
\begin{equation}
\label{generator}
\mathcal{G} f(x,i):= \frac{1}{2}\sigma^2 x^2 f_{xx}(x,i) + (r - g + \lambda_i)x f_x(x,i) + \sum_{j \neq i}q_{ij}\big[f(x,j)-f(x,i)\big].
\end{equation}
It is worth noting that, due to \eqref{generator}, equation \eqref{HJB} is actually a system of variational inequalities, coupled through the transition rates $q_{ij}$.

In what follows, we will use the optimal boundaries $a(\cdot)$ and $b(\cdot)$ of \eqref{bds}, which define the value function of the associated optimal stopping game in \eqref{valueOS} (equivalently, \eqref{valueOS2}), in order to construct the optimal debt ratio management policy for the original problem \eqref{eq:valueOC}.  

To that end, recall the boundaries $a(\cdot)$ and $b(\cdot)$ of \eqref{bds}, let $x\in [a(i),b(i)]$, $i\in \mathcal{M}$ and denote by $\widetilde{\mathcal{U}}$ the set of right-continuous adapted nondecreasing processes starting from $0$ at initial time. 
Then consider the two-sided Skorokhod reflection problem $\textbf{SP}(a,b;x,i)$ defined as: 
\begin{align}
\tag{$\textbf{SP}(a,b;x,i)$}
\hspace{-10pt}
\text{Find $(\widetilde\xi,\widetilde\eta) \in \widetilde{\mathcal{U}} \times \widetilde{\mathcal{U}}$ 
s.t.}
\left\{
\begin{array}{l}
\displaystyle X^{x,i,\widetilde\varphi}_{t}\in[a(Y_t),b(Y_t)],\,\, \text{$\P$-a.s.~for all $t>0$},\\[+9pt]
\displaystyle \int^{T}_0{\mathds{1}_{\{X^{x,i,\widetilde\varphi}_{t}>a(Y_t)\}}d\widetilde\xi_{t}}=0,\,\, \text{$\P$-a.s.~for any $T>0$,}\\[+9pt]
\displaystyle \int^{T}_0{\mathds{1}_{\{X^{x,i,\widetilde\varphi}_{t}<b(Y_t)\}}d\widetilde\eta_{t}}=0,\,\, \text{$\P$-a.s.~for any $T>0$,}
\end{array}
\right.
\end{align}
where we set $\widetilde\varphi := \widetilde\xi - \widetilde\eta$.
Such a problem admits a unique solution $(\widetilde{\xi}^\star,\widetilde{\eta}^\star)$; indeed, recalling \eqref{controlsbar} and \eqref{Xbar}, we can apply Proposition 2.3, Corollary 2.4 and Theorem 2.6 in \cite{Burdzyetal} by setting, in the notation of that paper, $\phi(t):=X^{x,i,\widetilde\varphi}_{t}/X^{x,i,0}_{t}$, $\psi(t):=x$, $\eta_{\ell}(t):=\int_0^t \frac{d\widetilde\xi_{s}}{X^{x,i,0}_{s}}$, $\eta_{r}(t):=\int_0^t \frac{d\widetilde\eta_{s}}{X^{x,i,0}_{s}}$, $\ell(t):=a(Y_t)/X^{x,i,0}_{t}$ and $r(t):=b(Y_t)/X^{x,i,0}_{t}$ 
(see also \cite{DAuria} for another example of a regime dependent Skorokhod problem).

We denote $\widetilde{\varphi}^\star:=\widetilde{\xi}^\star - \widetilde{\eta}^\star$ and we notice that $\text{supp}\{d\widetilde{\xi}^\star\}\cap \text{supp}\{d\widetilde{\eta}^\star\}=\emptyset$, since $a(i) < b(i)$ for any $i\in \mathcal{M}$ (see Proposition \ref{notemptyS1}).\
 Then, for any $(x,i) \in \mathcal{O}$ define the control (here and in the rest of the paper, $(\,\cdot\,)^+$ denotes the positive part) 
\begin{equation}
\label{eq:OC}
\begin{cases}
\; \varphi^{\star}:=\xi^{\star}-\eta^{\star} \quad \text{such that } \quad 
{\xi}^{\star}_0=0 ={\eta}^{\star}_0, \; \P-\text{a.s.}, \quad \text{where for any $t>0$,}\\
\; \xi^{\star}_t:= (a(i)-x)^{+} + \widetilde{\xi}^\star_{t-} \qquad \text{and} \qquad \eta^{\star}_t:= (x-b(i))^{+} + \widetilde{\eta}^\star_{t-} .
\end{cases}
\end{equation} 
The remaining of this section is dedicated to proving the optimality of the control \eqref{eq:OC} for the original debt ratio management problem \eqref{eq:valueOC}.

Before doing so, it is worth noticing that the debt ratio management policy prescribed by the controls in \eqref{eq:OC} involves two types of actions by the government: 
\vspace{0.15cm}
\\
$(a)$ {\it Small-scale actions} employed when the debt ratio $X_t$ approaches, at any time $t\geq 0$, either boundary $a(Y_t)$ from above or boundary $b(Y_t)$ from below. 
The purpose of these measures is to make sure (with a minimal effort) that the debt ratio level $X_t$ is kept inside the interval $[a(Y_t),b(Y_t)]$. Mathematically, these are the actions caused by the continuous parts ${\xi}^{\star,cont}$ and ${\eta}^{\star,cont}$ of the controls ${\xi}^\star$ and ${\eta}^\star$, respectively (Skorokhod reflection-type policies); 
\vspace{0.15cm}
\\
$(b)$ {\it Large-scale actions} employed when the debt ratio $X_t$, at any time $t\geq 0$, is either below the boundary $a(Y_t)$ or above the boundary $b(Y_t)$. The purpose of these  measures is to bring immediately the debt ratio level $X_t$ back inside the interval $[a(Y_t),b(Y_t)]$. Mathematically, these are the actions caused at time $t=0$, by the initial jumps $(a(i)-x)^{+}$ and $(x-b(i))^{+}$, or at any time $t>0$, by the jump parts $\Delta{\xi}^\star_{t}:={\xi}^\star_{t+}-{\xi}^\star_{t}$ and $\Delta{\eta}^\star_{t}:={\eta}^\star_{t+}-{\eta}^\star_{t}$ of the controls ${\xi}^\star$ and ${\eta}^\star$, respectively (Lump-sum-type policies). 

\begin{remark} 
\label{jumps}
Note that, the large-scale actions mentioned in $(b)$ above, caused by the jump parts $\Delta{\xi}^\star_{t}$ and $\Delta{\eta}^\star_{t}$ of the controls for $t>0$, will only be needed at times of jumps of the macroeconomic regime switching process $Y_t$. These are the only times when the debt ratio level $X_t$ may exit the interval $[a(Y_{t}),b(Y_{t})]$. This is an interesting feature, coming from the inclusion of regime switching macroeconomic factors in the model, not usually observed in bounded-variation stochastic control problems without regime switching, where a lump-sum action may be required only at time $t=0$ (see, e.g., \cite{GuoPham}, among others).
\end{remark}

In order to illustrate the argument in Remark \ref{jumps}, consider the following example.  Suppose that time $T$ is a jump time from the initial economic regime $Y_{T-}=i$ to a ``worse" one $Y_{T}=j$. Suppose also that, immediately before the jump, the debt ratio was inside the required bounds (i.e.\ $a(i) < X_{T-} < b(i)$), but after the jump it ends up above the new upper bound under the new regime $j$ (i.e.\ $a(j) <  b(j) < X_{T}$). In this case, the optimal debt ratio management policy of the government, which was ``just observing" (no-action) before the regime change, will now require a lump-sum type of austerity policy, e.g.\ with a large-scale spending cut, that can decrease the debt ratio level by $\Delta{\xi}^\star_{T}=X_T-b(j)$.  

\vspace{0.1cm}
We now proceed with the next lemma showing the admissibility of the control $\varphi^{\star}$ in \eqref{eq:OC}.
\begin{lemma}
\label{OCadm}
For any $(x,i) \in \mathcal{O}$, we have $\varphi^{\star}\in \mathcal{A}(x,i)$.
\end{lemma}
\begin{proof}
Clearly $\varphi^{\star} \in \mathcal{V}$. Also, for any $(x,i) \in \mathcal{O}$, we have $X^{x,i,\varphi^{\star}}_t \geq 0$, $\P$-a.s.\ for all $t\geq 0$ since $b(i)>a(i)>0$. It thus remains only to show that 
\begin{equation}
\label{integrability}
\E_{(x,i)}\bigg[\int_0^{\infty} e^{-\int_0^t \rho_{Y_s}ds} \big(d\xi^{\star}_t + d\eta^{\star}_t\big)\bigg] < \infty.
\end{equation}
Notice that \eqref{eq:OC} yields 
$$\E_{(x,i)}\bigg[\int_0^{\infty}e^{-\int_0^t \rho_{Y_s}ds} \big(d\xi^{\star}_t + d\eta^{\star}_t\big)\bigg] = (a(i)-x)^{+} + (x-b(i))^{+} + \E_{(z(x,i),i)}\bigg[\int_{0+}^{\infty} e^{-\int_0^t \rho_{Y_s}ds} \big(d\widetilde{\xi}^\star_t + d\widetilde{\eta}^\star_t\big)\bigg],$$
where $z(x,i)=x$ if $x \in (a(i),b(i))$, $z(x,i)=a(i)$ if $x\leq a(i)$ and $z(x,i)=b(i)$ if $x\geq b(i)$. Hence, to have \eqref{integrability} it suffices to prove that 
$$\E_{(z,i)}\bigg[\int_0^{\infty}e^{-\int_0^t \rho_{Y_s}ds} \big(d\widetilde{\xi}^\star_t + d\widetilde{\eta}^\star_t\big)\bigg]<\infty,$$
for any $z \in [a(i),b(i)]$. In the following we only prove that 
\begin{equation}
\label{adm1}
\E_{(z,i)}\bigg[\int_0^{\infty}e^{-\int_0^t \rho_{Y_s}ds} d\widetilde{\xi}^\star_t \bigg]<\infty, \quad (z,i) \in [a(i),b(i)]\times \mathcal{M},
\end{equation} 
since analogous arguments can be employed to show that $\E_{(z,i)}[\int_0^{\infty}e^{-\int_0^t \rho_{Y_s}ds} d\widetilde{\eta}^\star_t ]<\infty$. 

To prove \eqref{adm1} we adapt arguments from \cite{Shreveetal}. 
Let $\widetilde{X}:=X^{\widetilde{\varphi}^\star}$ and $g:\mathbb{R}\times \mathcal{M} \to \mathbb{R}$ be any solution to
$$\big(\mathcal{G}-\rho_i\big)g(x,i)=0.$$
Then, take a fixed $T>0$ and let $0 \leq T_1 < T_2 < ... < T_{M} \leq T$ be the random times of jumps of $Y$ in the interval $[0, T]$ (clearly, the number $M$ of those jumps is random as well). Notice that the times $T_n$, for $n=1,\ldots,M$, of regime changes are the only possible jump times of $\widetilde{\varphi}^\star$, as discussed in Remark \ref{jumps}.

By the regularity of $g$ we can apply It\^o-Meyer's formula for semimartingales (\cite{Meyer}, pp.\ 278--301) to the process $(e^{-\int_0^t \rho_{Y_s}ds} g(\widetilde{X}_{t}, Y_{t}))_{t\geq0}$ on each of the intervals $[0, T_1)$, $(T_1,T_2)$,...,$(T_M,T]$. Piecing together all the terms as in the proof of Lemma 3 at p.\ 104 of \cite{Sk}, we obtain
\begin{align}
\label{adm2}
& \E_{(z,i)}\Big[e^{-\int_0^T \rho_{Y_s}ds} g(\widetilde{X}_T,Y_T)\Big] - g(z,i) =  \E_{(z,i)}\bigg[\int_0^T e^{-\int_0^t \rho_{Y_s}ds} g_x(\widetilde{X}_t,Y_t)d\widetilde{\xi}^{\star,cont}_t\bigg] 
 \\ 
&- \E_{(z,i)}\bigg[\int_0^T \hspace{-.15cm} e^{-\int_0^t \rho_{Y_s}ds} g_x(\widetilde{X}_t,Y_t)d\widetilde{\eta}^{\star,cont}_t\bigg] 
+  \E_{(z,i)}\bigg[\sum_{0\le T_n \le T} \hspace{-.25cm} e^{-\int_0^{T_{n}} \rho_{Y_s}ds}\left(g(\widetilde{X}_{T_n}, Y_{T_n})-g(\widetilde{X}_{T_n-},Y_{T_n})\right)\bigg]. \nonumber
\end{align}
Observe that, the latter expectation in \eqref{adm2} can be written as
\begin{eqnarray}
\label{sum-jumps}
&& \E_{(z,i)}\bigg[\,\sum_{0\le T_n \le T}e^{-\int_0^{T_{n}} \rho_{Y_s}ds}\left(g(\widetilde{X}_{T_n}, Y_{T_n})-g(\widetilde{X}_{T_n-},Y_{T_n})\right)\bigg] \nonumber \\
&& = \E_{(z,i)}\bigg[\,\sum_{0\le T_n \le T} e^{-\int_0^{T_{n}} \rho_{Y_s}ds} \Big(\mathds{1}_{\{\Delta\widetilde{\xi}^\star_{T_n}>0\}}+\mathds{1}_{\{\Delta\widetilde{\eta}^\star_{T_n}>0\}}\Big)\left(g(\widetilde{X}_{T_n}, Y_{T_n})-g(\widetilde{X}_{T_n-},Y_{T_n})\right)\bigg]  \\
&& = \E_{(z,i)}\bigg[\,\sum_{0\le T_n \le T} e^{-\int_0^{T_{n}} \rho_{Y_s}ds} \Big( 
\int_0^{{\Delta\widetilde{\xi}^\star_{T_n}}}g_x(\widetilde{X}_{T_n-} + u, Y_{T_n}) du
- \int_0^{{\Delta\widetilde{\eta}^\star_{T_n}}}g_x(\widetilde{X}_{T_n-} - u, Y_{T_n}) du\Big)\bigg]. \nonumber
\end{eqnarray}

Impose now that $g_x(a(i),i)=-1$ and $g_x(b(i),i)=0$, and extend the function $g$ on $(-\infty,a(i)) \cup (b(i),\infty)$ so that $g_x(x,i)=-1$ for any $x < a(i)$ and $g_x(x,i)=0$ for any $x > b(i)$ (for example, set $g(x,i):=a(i)-x + g(a(i),i)$ for $x < a(i)$ and $g(x,i)=g(b(i),i)$ for $x>b(i)$). Then, since $\widetilde{\xi}^\star_{\cdot}$ is flat off $\{t\geq 0: \widetilde{X}_t \leq a(Y_t)\}$ and $\widetilde{\eta}_{\cdot}$ is flat off $\{t\geq 0: \widetilde{X}_t \geq b(Y_t)\}$ (cf.\ Problem $\textbf{SP}(a,b;z,i)$), we get 
\begin{align}
\label{xietacont}
\begin{cases}
g_x(\widetilde{X}_t,Y_t)d\widetilde{\xi}^{\star,cont}_t=-d\widetilde{\xi}^{\star,cont}_t 
\quad &\text{and } \quad 
g_x(\widetilde{X}_t,Y_t)d\widetilde{\eta}^{\star,cont}_t =0, \\
\int_0^{\Delta\widetilde{\xi}^\star_{T_n}}g_x(\widetilde{X}_{T_n-} + u, Y_{T_n}) du = - \Delta\widetilde{\xi}^\star_{T_n}
\quad &\text{and } \quad 
\int_0^{\Delta\widetilde{\eta}^\star_{T_n}}g_x(\widetilde{X}_{T_n-} - u, Y_{T_n}) du = 0.
\end{cases} 
\end{align}
Therefore, by substituting \eqref{xietacont} in \eqref{sum-jumps} and then \eqref{adm2}, we get that  
\begin{equation}
\label{adm2-bis}
\E_{(z,i)}\Big[e^{-\int_0^{T} \rho_{Y_s}ds} g(\widetilde{X}_T,Y_T)\Big] - g(z,i) = - \E_{(z,i)}\bigg[\int_0^T e^{-\int_0^{t} \rho_{Y_s}ds}  d\widetilde{\xi}^\star_t\bigg].
\end{equation}

Finally, given that $g(\widetilde{X}_T,Y_T) \leq \max_{i \in \mathcal{M}}\sup_{x \in [\min_j a(j), \max_j b(j)]}g(x,i)$, $\P_{(x,i)}$-a.s., we can let $T\uparrow \infty$, and apply the dominated convergence theorem on the left-hand side of \eqref{adm2-bis} and the monotone convergence theorem on its right-hand side, to obtain 
$$g(z,i) = \E_{(z,i)}\bigg[\int_0^{\infty} e^{-\int_0^{t} \rho_{Y_s}ds}  d\widetilde{\xi}^\star_t\bigg].$$
The finiteness of the function $g$ constructed above, yields \eqref{adm1}.
\end{proof}

Thanks to the admissibility of $\varphi^{\star}$ we can now prove its optimality.
\begin{theorem}
\label{thm:OC}
The admissible ${\varphi}^{\star}={\xi}^{\star} - {\eta}^{\star}$ of \eqref{eq:OC} is optimal for the problem \eqref{eq:valueOC}.
\end{theorem}
\begin{proof}
It suffices to show that $\mathcal{J}_{(x,i)}(\varphi^{\star}) = V(x,i)$ for any $(x,i) \in \mathcal{O}$. In order to simplify notation from now on we write $X^{\star}\equiv X^{\varphi^{\star}}$, $\P_{(x,i)}$-a.s.

Fix $(x,i)\in \mathcal{O}$, and take arbitrary $T>0$. Let $0 \leq T_1 < T_2 < ... < T_{M} < T$ be the random times of jumps of $Y$ in the interval $[0, T)$ (clearly, the number $M$ of those jumps is random as well). By the regularity of $V$ we can apply It\^o-Meyer's formula to the process $(e^{-\rho t} V(X^{\star}_{t}, Y_{t}))_{t\geq0}$ (see also proof of Lemma \ref{OCadm}), and taking expectations we get 
\begin{align}
\label{Ito-V-1}
&V(x,i) = \E_{(x,i)}\bigg[e^{-\int_0^{T} \rho_{Y_s}ds} V(X^{\star}_{T}, Y_{T})-\int_0^{T} e^{-\int_0^{t} \rho_{Y_s}ds} (\mathcal{G}-\rho)
V(X^{\star}_t,Y_t)dt\bigg]  \\ 
&- \E_{(x,i)}\bigg[\int_0^{T} \hspace{-0.1cm}e^{-\int_0^{t} \rho_{Y_s}ds} V_x(X^{\star}_t,Y_t)\Big(d\xi^{\star, cont}_t - d\eta^{\star, cont}_t\Big)\bigg] 
- \E_{(x,i)}\bigg[\sum_{0\le t < T} \hspace{-0.15cm} e^{-\int_0^{t} \rho_{Y_s}ds}
\left(V(X^{\star}_{t+}, Y_t)-V(X^{\star}_t,Y_t)\right)\bigg], \nonumber 
\end{align}
where we used the facts that the expectation of the stochastic integral vanishes since $X^{\star}_t \in [\min_{i}a(i), \max_{i}b(i)]$ and $V_x(\cdot,i)$ is continuous.

Recall now that $V$ solves \eqref{HJB} and $V_x = v$ by \eqref{Vx}, with $v$ as in \eqref{valueOS}. Hence, since $X^{\star}_t \in [a(Y_t),b(Y_t)]$, $\P_{(x,i)}$-a.s.\ for a.e.\ $t>0$, we have that $(\mathcal{G}-\rho_{Y_t})V(X^{\star}_t,Y_t)=-h(X^{\star}_t, Y_t)$ $\P_{(x,i)}$-a.s.\ for a.e.\ $t\geq0$. 
Furthermore, notice that $(\xi^{\star},\eta^{\star})$ solve the Skorokhod reflection problem, and therefore $\{t:\,d\xi^{\star}_t(\omega)>0\} \subseteq \{t:\,X^{\star}_t(\omega) \leq a(Y_t(\omega))\}$ and $\{t:\,d\eta^{\star}_t(\omega)>0\} \subseteq \{t:\,X^{\star}_t(\omega) \geq b(Y_t(\omega))\}$ for any $\omega \in \Omega$. Then, because $V_x(x,i) = c_2$ for $x \leq a(i)$ and $V_x(x,i) = c_1$ for $x \geq b(i)$, we obtain from \eqref{Ito-V-1} (see also \eqref{sum-jumps}) that
\begin{align}
\label{verif07}
V(x,i) =& \E_{(x,i)}\Big[e^{-\int_0^{T} \rho_{Y_s}ds} V(X^{\star}_{T}, Y_{T})\Big] 
+ \E_{(x,i)}\bigg[\int_0^{T} e^{-\int_0^{t} \rho_{Y_s}ds} h(X^{\star}_t, Y_t) dt +\int_0^{T} e^{-\int_0^{t} \rho_{Y_s}ds} \big(c_1 d\eta^{\star}_t - c_2 d\xi^{\star}_t\big) \bigg]. 
\end{align}

Since $X^{\star}_t \in [\min_{i}a(i), \max_{i}b(i)]$ and $V(\cdot,i)$ is continuous, applying the dominated convergence theorem gives 
$\lim_{T \uparrow \infty}\E_{(x,i)}[e^{-\int_0^{T} \rho_{Y_s}ds} V(X^{\star}_{T}, Y_{T})] = 0.$
Hence, taking limits as $T\to\infty$ in the second expectation on the right-hand side of \eqref{verif07}, and invoking the monotone convergence theorem, together with Lemma \ref{OCadm} and \eqref{setA}, we find
\begin{align*}
V(x,i) =& \E_{(x,i)}\bigg[\int_0^{\infty} e^{-\int_0^{t} \rho_{Y_s}ds} h(X^{\star}_t, Y_t) dt +\int_0^{\infty} e^{-\int_0^{t} \rho_{Y_s}ds} \big(c_1 d\eta^{\star}_t - c_2 d\xi^{\star}_t\big) \bigg]
= \mathcal{J}_{(x,i)}(\varphi^{\star}).
\end{align*}
The latter shows optimality of ${\varphi}^{\star}={\xi}^{\star} - {\eta}^{\star}$ and thus completes the proof.
\end{proof}

\begin{remark}
\label{rem:optimality}
Notice that the unique optimal debt ratio management policy $\varphi^{\star}$ from \eqref{eq:OC} is also optimal in the larger class of admissible controls $\big\{\varphi \in \mathcal{V}:\,\E[\int_0^{\infty} e^{-\int_0^{t} \rho_{Y_s}ds} \big(d\eta_t + d\xi_t\big)]< \infty\big\}$, when we allow for $X$ to become negative. In this paper we have however formulated the optimal debt management problem over the more economically relevant class $\mathcal{A}$.
\end{remark}

\section{Further Results in a Case Study}
\label{sec:case}

In this section we further develop our analysis in the case of regime switching only in the debt ratio dynamics. We henceforth assume that $\rho_i\equiv \rho$ (with $\rho:=\ol\rho = \ul\rho$) and $h(\cdot,i)\equiv h(\cdot)$ for all $i\in \mathcal{M}$.

\subsection{The Geometry of the State Space}
\label{sec:geometry}

In this subsection we study the geometry of the problem's state space. More precisely, we prove that the free boundaries $a(i)$ and $b(i)$ -- that are associated to the Dynkin game with value $v(x,i)$ (cf.\ Section \ref{sec:OSgame}) and trigger the optimal control rule -- admit a particular ordering across the different states of the economy.

Recall the Markov process $({\X}^0,Y)$ (cf.\ \eqref{newGBM})  of Section \ref{sec:OSgame}, and denote by $\mathcal{L}$ its infinitesimal generator as the second-order differential operator, acting for any $i\in \mathcal{M}$ on functions $u(\cdot,i) \in C^2(\R)$, given by
\begin{equation*}
\mathcal{L} u(x,i):= \frac{1}{2}\sigma^2 x^2 u_{xx}(x,i) + (r - g + \lambda_i + \sigma^2)x u_x(x,i) + \sum_{j \neq i}q_{ij}\big[u(x,j)-u(x,i)\big].
\end{equation*}
Then, from standard arguments based on the strong Markov property, and from Proposition \ref{prop:semih}, Proposition \ref{notemptyS1} and Theorem \ref{thm:reg-v},
it follows that for any $i\in \mathcal{M}$, the triplet $(v(\cdot,i),a(i),b(i))$ satisfies the following \emph{free-boundary problem}
\begin{align}
& \big(\mathcal{L}-\big(\rho-(r - g + \lambda_i)\big)\big)v(x,i) = - h_x(x),\,\,\qquad\,\,\,a(i) < x < b(i), \label{FBP-1} \\ 
& \big(\mathcal{L}-\big(\rho-(r - g + \lambda_i)\big)\big)v(x,i) \leq - h_x(x),\,\,\qquad\,\,\, x < b(i), \label{FBP-2} \\ 
& \big(\mathcal{L}-\big(\rho-(r - g + \lambda_i)\big)\big)v(x,i) \geq - h_x(x),\,\,\qquad\,\,\, x > a(i), \label{FBP-3} \\ 
& v(x,i) = c_2,\,\qquad\,\,\,\qquad \qquad   \qquad \qquad \qquad \qquad\quad x \leq a(i), \label{FBP-4} \\
& v(x,i) = c_1,\,\qquad\,\,\,\qquad \qquad  \qquad \qquad \qquad \qquad\quad x \geq b(i). \label{FBP-5} 
\end{align}
Moreover, $v(\cdot,i) \in C^1(\R_+)$ for any $i\in \mathcal{M}$ and $v_{xx}(\cdot,i) \in L^{\infty}_{\text{loc}}(\R_+)$ for any $i\in \mathcal{M}$.

\begin{proposition}
\label{prop:structure}
The following hold true:
\begin{itemize}
\item[(i)] $a(N) \geq a(N-1) \geq \dots \geq a(1)$ and $b(1) \leq b(2) \leq \dots \leq b(N)$;
\item[(ii)] $a(N) < b(1)$.
\end{itemize}
\end{proposition}
\begin{proof}
We prove the two parts separately.
\vspace{0.15cm}

\emph{Proof of (i).} From \eqref{valueOS2} it is easily seen that $v(x,1) \geq v(x,2) \geq \dots \geq v(x,N)$ since $\lambda_1 \geq \lambda_2 \geq \dots \geq \lambda_N$. This in particular implies that $\{x \geq 0: v(x,N) > c_2 \} \subseteq \dots \subseteq \{x \geq 0: v(x,2) > c_2\} \subseteq \{x \geq 0: v(x,1) > c_2\}$ and therefore, in view of \eqref{bds}, we know that $a(N) \geq a(N-1) \geq \dots \geq a(1)$. 

Analogous arguments show that $b(1) \leq b(2) \leq \dots \leq b(N)$.
\vspace{0.15cm}

\emph{Proof of (ii).} We argue by contradiction and we suppose that $b(1) < a(N)$. 

On one hand, any $x\in (b(1),a(N))$ is such that $x>b(1)>a(1)$ and $v(x,1)=c_1$ (cf.\ \eqref{bds}). Therefore \eqref{FBP-3} and \eqref{FBP-5} yield
\begin{equation}
\label{contra1}
-\big(\rho - \mu_1\big)c_1 + \sum_{j\neq 1}q_{1j}v(x,j) + q_{11}c_1 + h_x(x) \geq 0,
\end{equation} 
where we used the equality $\sum_{j\neq 1}q_{1j} = - q_{11}$ and set $\mu_1:=r +\lambda_1 - g$.

On the other hand, we also have that, any $x\in (b(1),a(N))$ is such that $x<a(N)<b(N)$ and $v(x,N)=c_2$ (cf.\ \eqref{bds}). Hence, \eqref{FBP-2} and \eqref{FBP-4} give
\begin{equation}
\label{contra2}
-\big(\rho - \mu_N\big)c_2 + \sum_{j\neq N}q_{Nj}v(x,j) + q_{NN}c_2 + h_x(x) \leq 0,
\end{equation}
where we used the equality $\sum_{j\neq N}q_{Nj} = - q_{NN}$ and set $\mu_N:=r +\lambda_N - g$.

In all, it follows from \eqref{contra1}--\eqref{contra2} that, for any $x\in (b(1),a(N))$, 
\begin{align}
\label{contra3}
F_N(x) &:=-\big(\rho - \mu_N\big)c_2 + \sum_{j\neq N}q_{Nj} v(x,j) + q_{NN}c_2 + h_x(x)  \nonumber \\
&\leq 0 \leq -\big(\rho - \mu_1\big)c_1 + \sum_{j\neq 1}q_{1j} v(x,j) + q_{11}c_1 + h_x(x)=:G_1(x).
\end{align}

Notice now that, by taking into account the inequalities $c_2 \leq v(x,j) \leq c_1$ for any $(x,j)\in \mathcal{O}$, together with Assumption \ref{ass:c12}, we obtain for any $x\in (b(1),a(N))$ that
\begin{align*}
G_1(x) \leq  -\big(\rho - \mu_1 \big) c_1 + h_x(x)  < -\big(\rho - \mu_N\big)c_2 + h_x(x) \leq F_N(x) , 
\end{align*}
which in view of \eqref{contra3} leads to a contradiction. 
\end{proof}

Proposition \ref{prop:structure} has the important consequence of characterising the geometry of continuation and stopping regions. This fact, combined with the regularity of the value function $v(\cdot,i)$ proved in Theorem \ref{thm:reg-v}, provides an operative method to determine the free boundaries $a(i)$ and $b(i)$, $i\in \mathcal{M}$. Indeed, since for any $i \in \mathcal{M}$ we have that $v(\cdot,i) \in C^1(\R_+)$, then $v(\cdot,i)$ must be necessarily continuously differentiable at the free boundaries $a(j)$ and $b(j)$ for all $j \in \mathcal{M}$. This yields the following system of nonlinear equations for the $2N$-dimensional vector $(a(1),b(1),\dots,a(N),b(N))$:
\begin{align}
\label{system-bds-1}
& v(a(i)+,i) = c_2 \quad \mbox{and} \quad v_x(a(i)+,i) = 0, \quad 
\forall\,\,i\in \mathcal{M}  \\
& v(b(i)-,i) = c_1 \quad \mbox{and} \quad v_x(b(i)-,i) = 0, \quad 
\forall\,\,i\in \mathcal{M} \label{system-bds-2} \\
& v(a(j)-,i) = v(a(j)+,i) \quad \mbox{and} \quad v_x(a(j)-,i) = v_x(a(j)+,i), \quad 
\forall\,\,(i,j)\in \mathcal{M}^2:\,\,j>i, \label{system-bds-3} \\
& v(b(j)-,i) = v(b(j)+,i) \quad \mbox{and} \quad v_x(b(j)-,i) = v_x(b(j)+,i), \quad 
\forall\,\,(i,j)\in \mathcal{M}^2:\,\,j<i. \label{system-bds-4}  
\end{align}

We will see how to explicitly write the system of equations for the boundaries in the following subsection, where we study the specific case in which the Markov chain $Y$ has $N=2$ states. Using the same steps, one can similarly write the associate system of equations for the boundaries in any other case of $N>2$.

\subsection{Explicit Solution in a Case Study with Two Regimes}
\label{sec:casestudy}

In this subsection, we consider the simplest possible regime switching model of debt ratio management. In particular, the continuous-time Markov chain $Y$, modelling the macroeconomic conditions affecting the interest rate on debt, has only $N = 2$ states; namely, $Y_{t} \in \mathcal{M}:=\{1,2\}$. In view of Assumption \ref{lY}, we have $\lambda_1 > \lambda_2$.
Therefore, the states $1$ and $2$ represent the ``bad'' and ``good'' scenarios for the government, under which the interest on debt is ``high'' and ``low'', respectively. 
We further assume a quadratic running cost function $h(x) = x^2/2$ for all $x>0$, which satisfies Assumption \ref{ass:h}-(i)--(iv); e.g.\ set $m=2$ and $K_1=K_2=K_3=1$ in Assumption \ref{ass:h}-(iii). 

Thanks to \ref{prop:Vx}, the government which originally aims at solving \eqref{eq:valueOC}, given by 
\begin{equation*}
V(x,i):=\inf_{\varphi \in \mathcal{A}} 
\E_{(x,i)}\bigg[\int_0^{\infty} e^{-\rho t} \frac{1}{2} \big(X^{\varphi}_t\big)^2 dt + c_1\int_0^{\infty} e^{-\rho t} d\eta_t - c_2\int_0^{\infty} e^{-\rho t} d\xi_t\bigg]
, \quad (x,i) \in \R_+ \times \{ 1,2 \} ,
\end{equation*}
can first find the value $v(x,i)$ of the optimal stopping game \eqref{valueOS} with \eqref{stfunct} and $\mathcal{O} \equiv \R_+ \times \{ 1,2 \}$. In view of \eqref{valueOS2}--\eqref{stfunct-bis}, $v(x,i)$ can be rewritten as
\begin{align}
\label{OS:example}
v(x,i) &= \sup_{\tau \geq 0}\inf_{\theta \geq 0} 
\EE_{(x,i)}\bigg[\int_{0}^{\tau \wedge \theta}e^{-\rrho_t} \X^{0}_t dt + c_2 e^{-\rrho_\tau}\mathds{1}_{\{\tau < \theta\}} + c_1 e^{-\rrho_\theta}\mathds{1}_{\{\theta < \tau\}}\bigg] \nonumber \\
&= \inf_{\theta \geq 0}\sup_{\tau \geq 0} 
\EE_{(x,i)}\bigg[\int_{0}^{\tau \wedge \theta}e^{-\rrho_t} \X^{0}_t dt + c_2 e^{-\rrho_\tau}\mathds{1}_{\{\tau < \theta\}} + c_1 e^{-\rrho_\theta}\mathds{1}_{\{\theta < \tau\}}\bigg] ,
\end{align}
for all $(x,i) \in \mathcal{O}$ and $\rrho_\cdot$ given by \eqref{rhohat}.
Then, the original value $V$ will follow from the equation \eqref{Vx} and the optimal debt ratio management policy given by \eqref{eq:OC} will involve the boundaries $a(1) \leq a(2) < b(1) \leq b(2)$ (cf.\ Proposition \ref{prop:structure}) that we obtain by solving \eqref{OS:example}.

\subsubsection{Derivation of the Explicit Solution.}
\label{sec:explsol}

In the following we write $q_1:=q_{12}=-q_{11}$ and $q_2:=q_{21}=-q_{22}$, as well as $k_i := \rho + q_i - 2 (r - g + \lambda_i) -\sigma^2 $ for both $i=1,2$.  
Equation \eqref{FBP-1}, used to obtain the value function $v(x,i)$ of the optimal stopping game, consists of the following coupled ordinary differential equations 
{\small
\begin{align*}
&\frac{1}{2}\sigma^2 x^2 v_{xx}(x,1) + (r - g + \lambda_1 +\sigma^2) x v_x(x,1) - (\rho - r + g - \lambda_1) v(x,1) + q_1 \big( v(x,2) - v(x,1) \big) = -x \\
&\frac{1}{2}\sigma^2 x^2 v_{xx}(x,2) + (r - g + \lambda_2+\sigma^2) x v_x(x,2) - (\rho - r + g - \lambda_2) v(x,2) + q_2 \big( v(x,1) - v(x,2) \big) = -x 
\end{align*}
}for all $a(1) < x < b(1)$ and $a(2) < x < b(2)$, respectively, while the value function should also satisfy the four conditions in \eqref{system-bds-1}--\eqref{system-bds-4} at the boundaries $a(i)$ and $b(i)$, for $i=1,2$ (see also the final paragraph of Section \ref{sec:geometry} for more details). 

Solving the system of ordinary differential equations we get that
\begin{align*} 
v(x,1) &= \begin{cases} c_2  &, \text{ if } x \leq a(1) ,\\ 
A_1 x^{\alpha_1} + A_2 x^{\alpha_2} + \frac{1}{k_1} x + \frac{c_2 q_1}{\rho + q_1 - (r - g + \lambda_1)} &, \text{ if } a(1) < x \leq a(2) , \\ 
B_1 x^{\beta_1} + B_2 x^{\beta_2} + B_3 x^{\beta_3} + B_4 x^{\beta_4} 
+ \frac{q_1 + k_2}{k_1 k_2 - q_1 q_2} x  &, \text{ if } a(2) < x \leq b(1) , \\ 
c_1 &, \text{ if } x \geq b(1) 
\end{cases}
\intertext{and} 
v(x,2) &= \begin{cases} c_2  &, \text{ if } x \leq a(2) \\
\frac{\Phi_1(\beta_1)}{q_1} B_1 x^{\beta_1} + \frac{\Phi_1(\beta_2)}{q_1} B_2 x^{\beta_2} + \frac{\Phi_1(\beta_3)}{q_1} B_3 x^{\beta_3} + \frac{\Phi_1(\beta_4)}{q_1} B_4 x^{\beta_4} 
+ \frac{k_1 + q_2}{k_1 k_2  - q_1 q_2} x   &, \text{ if } a(2) < x \leq b(1) , \\ 
C_1 x^{\gamma_1} + C_2 x^{\gamma_2} + \frac{1}{k_2} x + \frac{c_1 q_2}{\rho + q_2 - (r - g + \lambda_2)}  &, \text{ if } b(1) < x \leq b(2) , \\ 
c_1  &, \text{ if } x \geq b(2) ,
\end{cases}
\end{align*}
where the constants $\alpha_2 < 0 < \alpha_1$ (under Assumption \ref{ass:rho} we have $\alpha_1 > 1$) are given by
{\small
\begin{equation*}
\alpha_{1,2} = \frac{1}{2} + \frac{r - g + \lambda_1}{\sigma^2} \pm 
\sqrt{ \left( \frac{1}{2} + \frac{r - g + \lambda_1}{\sigma^2} \right)^2 + \frac{2 \big( \rho + q_1 - (r - g + \lambda_1) \big)}{\sigma^2} }  ,
\end{equation*} 
}the constants $\gamma_2 < 0 < \gamma_1$ (under Assumptions \ref{lY} and \ref{ass:rho} we have $\gamma_1 > 1$) are given by
{\small 
\begin{equation*}
\gamma_{1,2} = \frac{1}{2} + \frac{r - g + \lambda_2}{\sigma^2} \pm 
\sqrt{ \left( \frac{1}{2} + \frac{r - g + \lambda_2}{\sigma^2} \right)^2 + \frac{2 \big( \rho + q_2 - (r - g + \lambda_2) \big)}{\sigma^2} },  
\end{equation*} 
}and the constants $\beta_4 < \beta_3 < 0 < \beta_2 < \beta_1$ are the solutions of the characteristic equation $\Phi_1(\beta) \, \Phi_2(\beta) = q_1 \, q_2$ with
{\small
\begin{equation*}
\Phi_i(\beta) = \frac{1}{2} \sigma^2 \beta^2 + \Big( r - g + \lambda_i + \frac{1}{2} \sigma^2 \Big) \beta - \big( \rho + q_i - (r - g + \lambda_i) \big) , \quad \text{for } i=1,2.
\end{equation*}
}

Then, applying the conditions in \eqref{system-bds-1} and \eqref{system-bds-2} at the boundaries $a(i)$ and $b(i)$, for $i=1,2$, we obtain the following expressions
{\small
\begin{align} \label{Ai}
A_i \equiv A_i\big( a(1) \big) &= \frac{(-1)^{i+1} a^{- \alpha_i}(1)}{\alpha_1 - \alpha_2} \bigg[
\frac{\alpha_{3-i} -1}{k_1} \, a(1) - \frac{\alpha_{3-i} c_2 \big( \rho - (r - g + \lambda_1) \big)}{\rho + q_1 - (r - g + \lambda_1)} \bigg]  \,,\\
\label{Ci}
C_i \equiv C_i\big( b(2) \big) &= \frac{(-1)^{i+1} b^{- \gamma_i}(2)}{\gamma_1 - \gamma_2} \bigg[
\frac{\gamma_{3-i} -1}{k_2 } \, b(2) - \frac{\gamma_{3-i} c_1 \big( \rho - (r - g + \lambda_2) \big)}{\rho + q_2 - (r - g + \lambda_2)} \bigg] \,,
\end{align}
}for $i=1,2$, as well as 
{\small 
\begin{align} \label{Bi}
B_i &\equiv B_i\big( a(2), b(1) \big) \qquad \qquad \qquad \qquad \qquad \qquad \qquad \qquad \qquad \qquad \qquad \qquad \qquad \qquad \qquad \qquad \qquad \\ 
&= \frac{
\sum_{\substack{ j,k,l \in \mathcal{I} \setminus \{i\} : \\ l \not= j<k \not= l}} \hspace{0pt} (-1)^{k-j+\mathds{1}_{\{l>i\}}} (\beta_j - \beta_k) \bigg[ \frac{\Phi_1(\beta_j) \Phi_1(\beta_k)}{q_1^2}  f_{l,1}\big(b(1)\big) \Big( \frac{a(2)}{b(1)}\Big)^{\beta_j + \beta_k} \hspace{-4pt}+ \frac{\Phi_1(\beta_l)}{q_1}  f_{l,2}\big(a(2)\big) \Big( \frac{a(2)}{b(1)}\Big)^{\beta_l} \bigg] 
}{
b^{\beta_i}(1) \sum_{\substack{j,k,l \in \mathcal{I} \setminus \{1\} : \\ j \not= k<l \not= j}} 
\hspace{0pt} (-1)^{j+1} (\beta_1 - \beta_j) (\beta_k - \beta_l) \bigg[  \frac{\Phi_1(\beta_1) \Phi_1(\beta_j)}{q_1^2} \Big( \frac{a(2)}{b(1)}\Big)^{\beta_1 + \beta_j} \hspace{-4pt}+ \frac{\Phi_1(\beta_k) \Phi_1(\beta_l)}{q_1^2} \Big( \frac{a(2)}{b(1)}\Big)^{\beta_k + \beta_l} \bigg] 
} \nonumber
\end{align}
}for $i \in \mathcal{I} := \{1,2,3,4\}$ and 
{\small
\begin{align*}
f_{m,n}(x) 
&= \frac{(1-\beta_m) ( k_{3-n} + q_n) \, x }{k_1 k_2  - q_1 q_2}+ \beta_m c_n
\end{align*}
}for $m \in \mathcal{I}$ and $n=1,2$. Notice that under Assumption \ref{ass:rho} all the denominators in the formulas above are nonzero.

We then apply \eqref{system-bds-3}--\eqref{system-bds-4} and we obtain
\begin{align*}
&v(a(2)+,1) = v(a(2)-,1) \quad \& \quad v_x(a(2)+,1) = v_x(a(2)-,1), \\ 
&v(b(1)+,2) = v(b(1)-,2) \quad\, \& \quad v_x(b(1)+,2) = v_x(b(1)-,2).
\end{align*}
Using the above conditions for the expressions of $v(x,i)$ for $i=1,2$ with $A_i$, $C_i$ for $i=1,2$ and $B_i$ for $i=1,2,3,4$ given by \eqref{Ai}--\eqref{Bi}, we obtain the boundaries $a(i)$ and $b(i)$ for $i=1,2$ as the solution of the following system of four arithmetic equations:  
{\small
\begin{align}
\label{System-1}
&\sum_{i=1}^{2} A_i\big( a(1) \big) a^{\alpha_i}(2) = \sum_{i=1}^{4} B_i \big( a(2), b(1) \big) \, a^{\beta_i}(2) + \frac{q_1 \big( f_{1,2}\big( a(2)\big) - \beta_1 c_2 \big)}{(1- \beta_1) k_1} 
-\frac{q_1c_2}{\rho + q_1 - (r - g + \lambda_1)} 
\\
&\sum_{i=1}^{2} \alpha_i A_i\big( a(1) \big) a^{\alpha_i}(2) = \sum_{i=1}^{4} \beta_i B_i \big( a(2), b(1) \big) \, a^{\beta_i}(2) + \frac{q_1 \big( f_{1,2}\big( a(2)\big) - \beta_1 c_2 \big)}{(1- \beta_1)k_1} \label{System-2} \\
\label{System-3}
&\sum_{i=1}^{2} C_i\big( b(2) \big) b^{\gamma_i}(1) = \sum_{i=1}^{4} \frac{\Phi_1(\beta_i)}{q_1} B_i \big( a(2), b(1) \big) b^{\beta_i}(1) + \frac{q_2  \big( f_{1,1}\big( b(1)\big) - \beta_1 c_1 \big)}{(1- \beta_1)k_2} 
-\frac{q_2 c_1}{\rho + q_2 - (r - g + \lambda_2)}  
\\
&\sum_{i=1}^{2} \gamma_i C_i\big( b(2) \big) b^{\gamma_i}(1) = \sum_{i=1}^{4} \beta_i \frac{\Phi_1(\beta_i)}{q_1} B_i \big( a(2), b(1) \big) b^{\beta_i}(1) + \frac{q_2  \big( f_{1,1}\big( b(1)\big) - \beta_1 c_1 \big)}{(1- \beta_1) k_2} \label{System-4}
\end{align}
} 

Finally, for any $i=1,2$, combining \eqref{FBP-2} with \eqref{FBP-4}, and \eqref{FBP-3} with \eqref{FBP-5}, we find that the boundaries $a(1), a(2), b(1), b(2)$ must necessarily be such that 
\begin{equation}
\label{bound-1}
x - \big(\rho + q_i - (r-g+\lambda_i)\big)c_2 + q_i v(x,j) \leq 0, \quad \mbox{for $j\neq i$ and $x<a(i)$},
\end{equation}
and
\begin{equation}
\label{bound-2}
x - \big(\rho + q_i - (r-g+\lambda_i)\big)c_1 + q_i v(x,j) \geq 0, \quad \mbox{for $j\neq i$ and $x>b(i)$}.
\end{equation}
The above conditions have the practical use of providing bounds on $a(1), a(2), b(1), b(2)$ that one has to check on a case by case basis when trying to solve numerically \eqref{System-1}--\eqref{System-4}.

It is worth stressing that one advantage of our direct probabilistic method -- compared to the traditional analytic guess-and-verify one -- is that existence of a solution to \eqref{System-1}--\eqref{System-4} satisfying \eqref{bound-1}--\eqref{bound-2} does not have to be proved, since it follows directly from the general theory developed in Section \ref{sec:OSgame}, in particular Theorem \ref{thm:reg-v} and Proposition \ref{prop:structure}. Moreover, we also have \emph{uniqueness} of such a solution. Indeed, if there were another quadruple $(\widetilde{a}(1), \widetilde{a}(2), \widetilde{b}(1), \widetilde{b}(2))$ solving \eqref{System-1}--\eqref{System-4} and satisfying \eqref{bound-1}--\eqref{bound-2}, by a standard verification argument one could prove that the bounded variation control that keeps the process $(X_t,Y_t)$ in the region $\{(x,i) \in \mathcal{O}:\, \widetilde{a}(i) \leq x \leq \widetilde{b}(i) \}$ for almost every $t\geq 0$ (i.e.\ solving $\textbf{SP}(\widetilde{a},\widetilde{b};x,i)$) is optimal. However, this would contradict the uniqueness of the optimal control proved in Theorem \ref{thm:existence}. 

\begin{remark}
\label{rem:complex}
Here we comment on the structure of the value function in the general case of $N\geq2$ regimes. 

In the above case study with $N=2$ regimes, there are $4$ boundaries $a(i), b(i), i=1,2$, solving uniquely the system of $4$ algebraic equations with constraints in \eqref{System-1}--\eqref{bound-2}, and the value function involves in total $8$ boundary-dependent-coefficients given by \eqref{Ai}--\eqref{Bi}. 

When solving the problem with $N$ regimes, the expression of the value function in each of the subintervals of the $i$-section of the continuation region $\mathcal{C}^i=\{x\geq 0:\, a(i) < x < b(i)\}$, for any $i\in \mathcal{M}$, will again have two components. The first component is the particular solution to the coupled system of $N$ ordinary differential equations (cf.\ \eqref{FBP-1}), and it will always be a linear function with coefficients depending only on the parameters of the problem. The second component is the general solution to the coupled system of $N$ ordinary differential equations, and it will be a polynomial with coefficients (in total the value function will involve $2N^2$ such coefficients) depending on the $2N$ boundaries (in total) of the continuation region. The latter boundaries will uniquely solve a system of $2N$ algebraic equations with constraints. 

It is then clear that for large $N$ the complexity of the problem makes its analysis a daunting task. However, by tackling the considered problem with our direct probabilistic approach, allows one to obtain important information about the structure and the regularity of the value function, as well as the geometry of the state space. Therefore, what remains to be done is just to , ind the numerical solution to the system of $2N$ algebraic equations discussed above.
\end{remark}

\subsubsection{Comparative Statics Analysis}
\label{sec:compstat}

In this subsection we show how the optimal control boundaries $a(1)$, $a(2)$, $b(1)$, $b(2)$, which define the government's debt ratio management policy, depend on the relevant model's parameters, and we provide interpretations of the results. 
In what follows, whenever we need to stress the dependence of the boundaries and value function on a given parameter $\chi$, we will write $a(i;\chi)$ and $b(i;\chi)$, as well as $v(x,i;\chi)$, $x\geq 0$ and $i=1,2$.

Our analysis begins with a theoretical proof of the monotonicity of the control boundaries with respect to $r-g$, and a numerical illustration in Figure \ref{Figure-1}. We then continue with a numerical study of the sensitivity with respect to $\sigma$ and $q_2-q_1$. Due to the complexity of our problem, proving analytically the monotonicity of $a(i)$ and $b(i)$, $i=1,2$, with respect to  $\sigma$ and $q_2-q_1$ is far from trivial. However, the explicit nature of our results (cf.\ the system of equations \eqref{System-1}--\eqref{System-4}) allows for an easy numerical implementation resulting in Figure \ref{Figure-2} and Figure \ref{Figure-3}.
\vspace{0.25cm}
 
\textit{Comparative Statics with respect to $r-g$.} We start with the following result.
\begin{proposition}
\label{prop:r-g}
For any $i\in \{1,2\}$ we have that $(r-g) \mapsto a(i;r-g)$ and $(r-g) \mapsto b(i;r-g)$ are decreasing.
\end{proposition}
\begin{proof}
Let $i \in \{1,2\}$ be given and fixed. Remember that from \eqref{bds} we can write
\begin{align*}
&a(i;r-g)=\inf\{x \geq 0:\,v(x,i;r-g)>c_2\}, \\
&b(i;r-g)=\sup\{x \geq 0:\,v(x,i;r-g)<c_1\}.
\end{align*}
From \eqref{OS:example} it is easily seen that $(r-g) \mapsto v(x,i;r-g)$ is increasing. Hence, \eqref{bds} imply that $(r-g) \mapsto a(i;r-g)$ and $(r-g) \mapsto b(i;r-g)$ are decreasing, and the claim thus follows.
\end{proof}

\begin{remark}
It is worth noticing that the proof of the previous result does not use the fact that the continuous-time Markov chain $Y$ has only two states. Therefore, Proposition \ref{prop:r-g} does hold in the more general setting of $N\geq2$. 
\end{remark}

It is clear from \eqref{freeGBM} that the higher the real interest rate on debt (net of the GDP growth rate), the more the country's debt ratio increases in expectations. In such a case, the result of Proposition \ref{prop:r-g} implies that the government should adopt a more restrictive policy for the management of public debt, in order to dam the resulting expected costs. 
In other words, as $r-g$ increases, the critical level, below which the government aims at keeping the debt ratio, decreases, so that the government should (optimally) intervene sooner to reduce the debt ratio, through austerity policies in the form of spending cuts. On the other hand, the trigger level at which the government starts increasing the debt ratio decreases as well, meaning that the government should be willing to postpone its public investment intervention which increases the debt ratio. (see Figure \ref{Figure-1}). 

We can also observe from Figure \ref{Figure-1} that when the interest rate on debt $r$ is sufficiently higher that the GDP growth rate $g$, then the debt ratio ceiling values $b(1)$ and $b(2)$ seem to come closer, thus implying that the debt reduction policy is not strongly affected by the state of the economy. Similarly, the trigger values $a(1)$ and $a(2)$ seem to converge to each other when the GDP grows at a much higher rate than the interest on debt. Hence under such a high GDP growth, the government can adopt, independently of the economic regime, a similar policy for public investments, aiming at increasing the debt ratio.
On the contrary, the trigger levels $a(1)$ and $a(2)$ (resp.\ $b(1)$ and $b(2)$) take significantly different values when $g$ is sufficiently lower than $r$ (resp.\ $r$ is sufficiently lower than $g$), so that in this case the debt policy seems to strongly react to the state of the economy.

Furthermore, under the choice of parameters of Figure \ref{Figure-1}, the levels $b(i)$, $i=1,2$, that trigger the debt reduction policies are on average equal to $60\%$, a value in line with the Maastricht Treaty's reference value of 1992.

\begin{figure}
\includegraphics[scale=0.5]{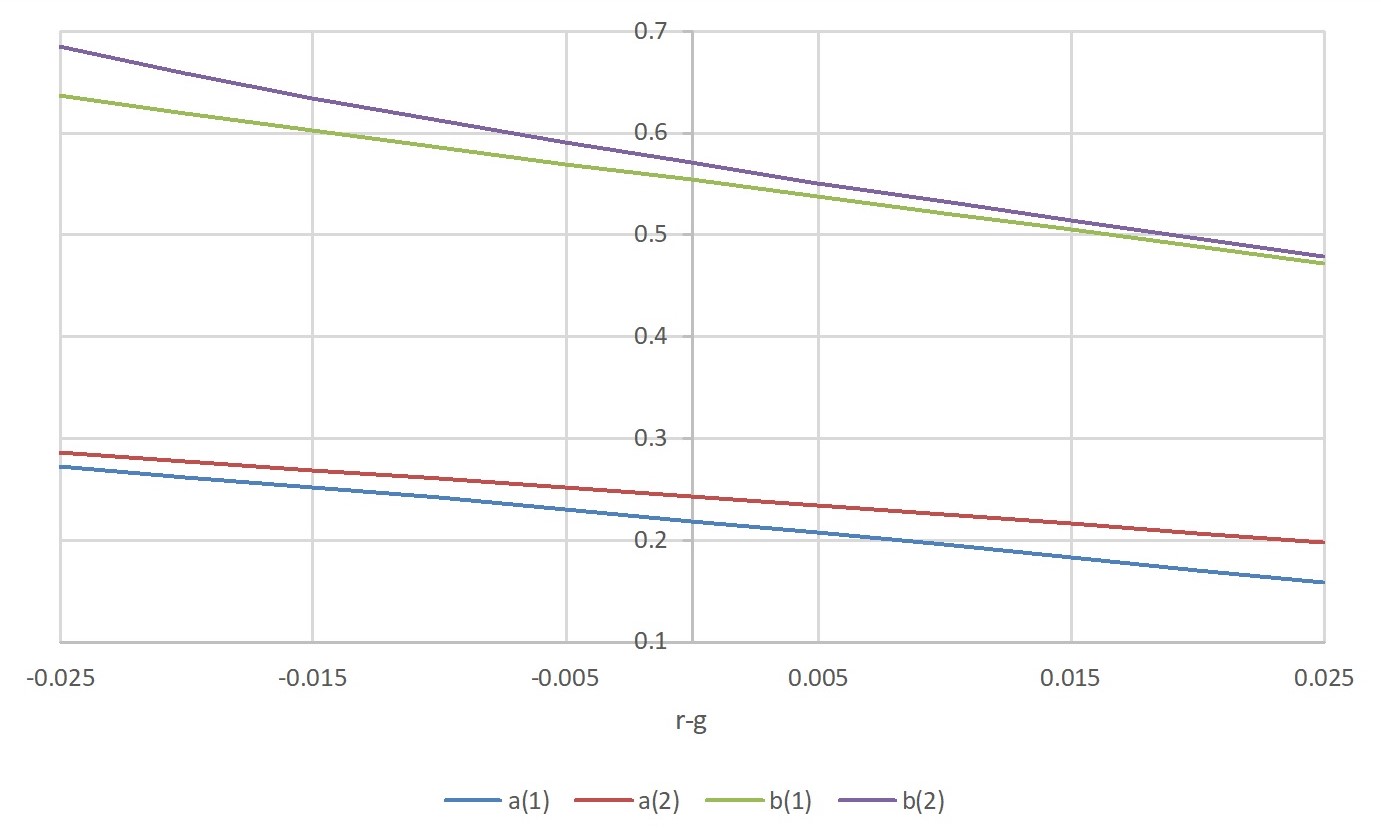}
\caption{Monotonicity of the control boundaries for $i=1,2$ with respect to $r-g$. For this plot we have used the following parameters' values: $q_1=0.02$, $q_2=0.02$, $\lambda_1=0.1$, $\lambda_2=0$, $\sigma=0.15$, $\rho=0.25$, $c_1=2$, $c_2=1.25$.}
\label{Figure-1}
\end{figure}
\vspace{0.25cm}

\textit{Comparative Statics with respect to $\sigma$.} We now move on to the study of the sensitivity of the control boundaries with respect to the debt ratio's volatility $\sigma$. We can observe from Figure \ref{Figure-2} that, in both regimes $i=1$ and $i=2$, the amplitude of continuation region $b(i)-a(i)$ increases with $\sigma$. This result is well known in the literature on real options (see \cite{DixitPind}, among others). In our setting of the debt ratio management, this means that the more volatile the debt ratio, the more cautious the government is, hence the longer it should wait before intervening on the debt ratio.

\begin{figure}
\includegraphics[scale=0.5]{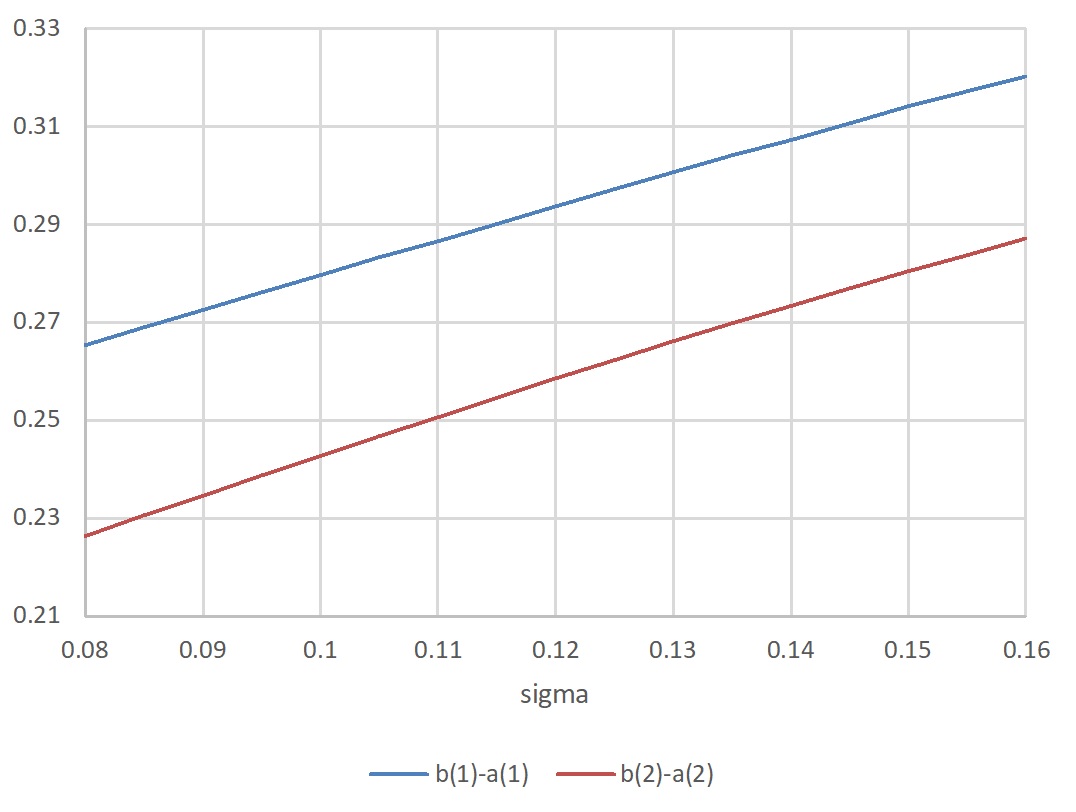}
\caption{Monotonicity of the continuation (no-action) region's size $b(i)-a(i)$, $i=1,2$, with respect to $\sigma$. For this plot we have used the following parameters' values: $q_1=0.02$, $q_2=0.02$, $r=0.04$, $g=0.015$, $\lambda_1=0.1$, $\lambda_2=0$, $\rho=0.25$, $c_1=2$, $c_2=1.25$.}
\label{Figure-2}
\end{figure}
\vspace{0.25cm}

\textit{Comparative Statics with respect to $q_2-q_1$.} It is seen in Figure \ref{Figure-3} that, in both regimes $i=1$ and $i=2$, the amplitude of the continuation region $b(i)-a(i)$ decreases when $q_2 - q_1$ increases. In particular, this can be viewed in two ways: On one hand, when the economy is in the ``bad'' state $i=1$, a decreasing rate $q_1$ of moving to the ``good'' regime $i=2$, suggests that the government should become more proactive, adopt a more restrictive policy and be willing to intervene more frequently on the debt ratio. This will counterbalance the fact that it is expected to remain under the ``bad'' regime for a longer time. On the other hand, when the economy is in the ``good'' state $i=2$, an increasing rate $q_2$ of moving to the ``bad'' regime $i=1$, suggests that the government should again become more proactive by adopting a more restrictive policy, so that it is more prepared to deal with the worse economic scenario.

\begin{figure}
\includegraphics[scale=0.7]{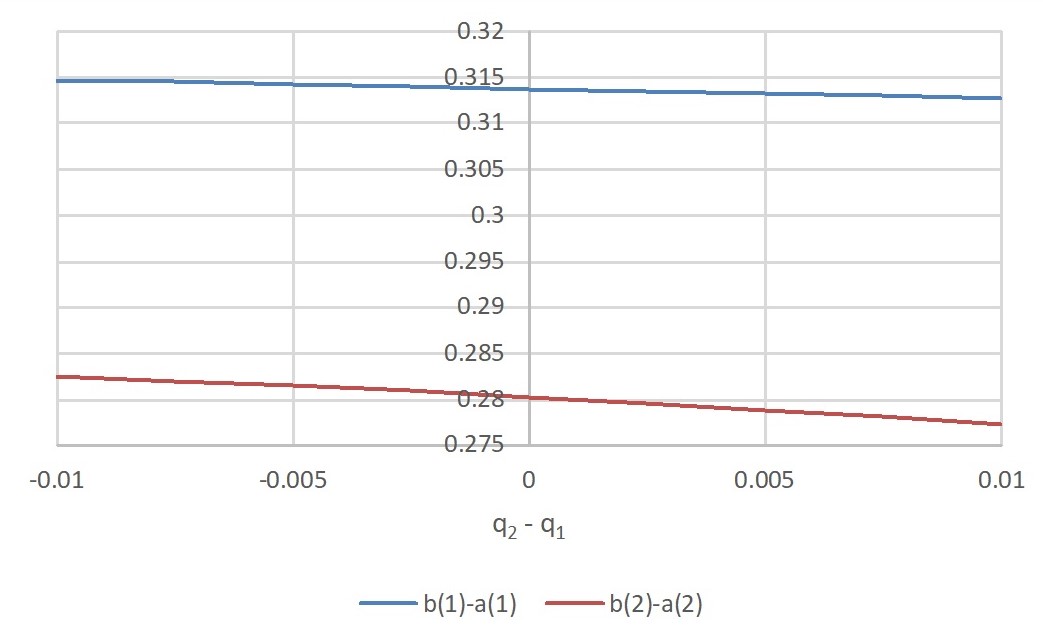}
\caption{Monotonicity of the continuation region's size, under both regimes, with respect to $q_2-q_1$. For this plot we have used the following parameters' values: $r=0.04$, $g=0.015$, $\lambda_1=0.1$, $\lambda_2=0$, $\sigma=0.15$, $\rho=0.25$, $c_1=2$, $c_2=1.25$.}
\label{Figure-3}
\end{figure}

\subsection{Comparison with the no-regime-switching case}
\label{sec:comparison}

In this section, we first present the solution to the no-regime-switching case, namely, the problem with only one regime $N=1$. Then, we compare the resulting optimal government policy with the regime switching optimal policy from Section \ref{sec:compstat} (where $N=2$) and we comment on the results. 

Observe that, under no-regime-switching, the dynamics of the governmentally managed debt-to-GDP ratio become one-dimensional and read as (compare with \eqref{DGDP})
\begin{equation*}
d\widetilde{X}_t=\big(r - g\big) \widetilde{X}_t dt + \sigma \widetilde{X}_t dW_t + d\xi_t - d\eta_t, \quad t>0, \qquad \widetilde{X}_0=x \in \mathbb{R_+} ,
\end{equation*} 
where we assume there is no additional macroeconomic risk process $Y$, in the form of a continuous-time Markov chain, and the (constant) interest rate on debt is simply given by the parameter $r$. 
In this case, the debt ratio management problem \eqref{eq:J}--\eqref{eq:valueOC} becomes one-dimensional as well, i.e.\ $V(x,i) \equiv V(x)$. 
Moreover, the boundaries involved in the two-sided Skorokhod reflection problem $\textbf{SP}(a,b;x,i) \equiv \textbf{SP}(a,b;x)$, defining the optimal controls in \eqref{eq:OC} and consequently the optimal policy of the government, are also constants denoted by $a$ and $b$. 

It follows from standard theory on singular stochastic control problems (see Chapter VIII in \cite{FlemingSoner}; compare also with the related problem in \cite{GuoPham}, among others) that the value function $V$ of \eqref{eq:valueOC} with $h(x) = x^2 / 2$ in \eqref{eq:J}, satisfies the following ordinary differential equation with boundary conditions:
\begin{align*}
&\frac{1}{2}\sigma^2 x^2 V_{xx}(x) + (r - g)x V_x(x) - \rho V(x) = - \frac{1}{2} x^2 
\quad \text{for } a < x < b , \\
&V_x(a+) = c_2 \quad \text{and} \quad V_x(b-) = c_1 , \\
&V_{xx}(a+) = 0 \quad \text{and} \quad V_{xx}(b-) = 0 . 
\end{align*}

Solving the above free-boundary problem, and imposing continuity of $V$ at $x=a$ and $x=b$, we get that 
\begin{align*} 
V(x) &= \begin{cases} V(a) - c_2 \, (a-x)    &, \text{ if } x \leq a ,\\ 
D_1 x^{\delta_1} + D_2 x^{\delta_2} + \frac{1}{2 (\rho - 2 (r - g) - \sigma^2)} x^2 &, \text{ if } a < x < b , \\ 
V(b) + c_1 \, (x-b) &, \text{ if } x \geq b ,
\end{cases}
\end{align*} 
with 
\begin{align*} 
D_i \equiv D_i(a,b) = \frac{\big(a - c_2 (\rho - 2 (r - g) - \sigma^2) \big) \big( \frac{b}{a}\big)^{\delta_{3-i}} - \big(b - c_1 (\rho - 2 (r - g) - \sigma^2)\big) \big( \frac{b}{a}\big)}
{(-1)^{i+1} \, \delta_i \, (\rho - 2 (r - g) - \sigma^2) \, a^{\delta_i - 1} \, \Big[ \big( \frac{b}{a}\big)^{\delta_1} - \big( \frac{b}{a}\big)^{\delta_2} \Big]} ,
\end{align*}
where the constants $\delta_2 < 0 < 1 < \delta_1$ are given by
\begin{equation*}
\delta_{1,2} = \frac{1}{2} - \frac{r - g}{\sigma^2} \pm 
\sqrt{ \left( \frac{1}{2} - \frac{r - g}{\sigma^2} \right)^2 + \frac{2 \rho}{\sigma^2} }  \,.
\end{equation*}
and the optimal boundaries $a \leq c_2 (\rho - r + g) < c_1 (\rho - r + g) \leq b$ are given by the unique solution to the system of arithmetic equations 
\begin{align*} 
&J_{1,2}(a) = J_{1,1}(b) \quad \text{and} \quad J_{2,2}(a) = J_{2,1}(b)
\end{align*}
where 
\begin{equation*}
J_{i,j}(x) = \frac{(\delta_i - 2) \, x - c_j (\delta_i -1) (\rho - 2 (r - g) - \sigma^2)}{x^{\delta_{3-i} - 1}} \,.
\end{equation*}
\vspace{0.15cm}

In order to compare the governmental optimal policy when there is no regime switching with the case study with $N=2$ regimes, we numerically calculate the values of the boundaries $a$ and $b$ and compare with the values of $a(1), a(2), b(1)$ and $b(2)$. 
Recall that, the no-regime-switching case assumes a constant interest rate $r$. Thus, in order to facilitate the comparison, we assume that under the ``good" economic regime $i=2$ in the two-regime case, we set $\lambda_2=0$, so that it also corresponds to an interest rate on debt equal to $r$. Then, under the ``bad" economic regime $i=1$, the interest rate on debt becomes $r+\lambda_1>r$; see Table \ref{Tab1}.

If there is a possibility for the government to experience different economic regimes, it is seen from Table \ref{Tab1} that the government should become more proactive, by adopting a more restrictive debt reduction policy. 
Even under the ``good" economic regime $i=2$, the government should (optimally) intervene sooner through austerity policies to reduce the debt ratio (at $58.23\%$), as opposed to the consistently ``good" economy under no regime switching, where the government is willing to intervene at a later stage (at $60.34\%$). This occurs irrespective of the fact that all parameters take exactly the same values. Clearly, the possibility of a future turn of events, leading to worse macroeconomic conditions, is what makes the government more cautious about the future and willing to intervene more frequently so that it is more prepared to deal with the worse economic scenario if and when it comes.
This also results in the slight postponing of public investments under the possibility of such change from $i=2$ to the worse economic regime $i=1$ (at a safer level $24.76\%$) compared to the slightly higher trigger level, when the  economy is consistently at a ``good" state (at $24.85\%$).   

\begin{table}[H] 
\begin{tabular}{ | c | c | c | c | } 
\hline
Number of &  \multirow{2}{1.3cm}{Regime} & \multicolumn{2}{c|}{Optimal boundaries (in $\%$)} \\
\cline{3-4}
Regimes & & $\quad \;\; \quad a \quad \quad \;\;$ & $\quad \quad b \quad \quad$ \\
\hline
\multirow{2}{1.1cm}{$N=2$} & $i=1$ & 22.5871 & 56.3248 \\
\cline{2-4}
& $i=2$ & 24.7630 & 58.2346 \\
\hline
$N=1$ & & 24.8539 & 60.3393 \\ 
\hline
\end{tabular} 

\bigskip
\caption{For this table we used the following parameter values: 
$r=0.012$, $g=0.015$, $\sigma=0.15$, $\rho=0.25$, $c_1=2$, $c_2=1.25$; 
and, for the $N=2$ case, the additional parameter's values: 
$\lambda_1=0.1$, $\lambda_2=0$, $q_1=0.02$, $q_2=0.02$.}

\label{Tab1}
\end{table}

\section*{Acknowledgments}

\noindent Financial support by the German Research Foundation (DFG) through the Collaborative Research Centre 1283 ``Taming uncertainty and profiting from randomness and low regularity in analysis, stochastics and their applications'' is gratefully acknowledged by Giorgio Ferrari. 

Financial support by the EPSRC via the grant EP/P017193/1 ``Optimal timing for financial and economic decisions under adverse and stressful conditions'' is gratefully acknowledged by Neofytos Rodosthenous.

We thank anonymous referees and associate editor for valuable comments and suggestions. Moreover, we are grateful to Dr.\ Gerardo Ferrara from Bank of England for fruitful discussions.


\begin{thebibliography}{99}

\bibitem{Billingsley} \textsc{Billingsley, P.}\ (1999). \emph{Convergence of Probability Measures}. 2nd Edition. Wiley.

\bibitem{BlanchardFischer}\textsc{Blanchard, O., Fischer, S.}\ $(1989)$. \emph{Lectures in Macroeconomics}. Cambridge, MA and London: MIT Press.

\bibitem{BouchardTouzi} \textsc{Bouchard, B., Touzi, N.}\ (2011). \emph{Weak Dynamic Programming Principle for Viscosity Solutions}. SIAM J.~Control Optim.~\textbf{49(3)}, pp.\ 948--962.

\bibitem{Burdzyetal} \textsc{Burdzy, K., Kang, W., Ramanan, K.}\ (2009). \emph{The Skorokhod Problem in a Time-Dependent Interval}. Stoch.\ Process.\ Appl.\ \textbf{119}, pp.\ 428--452.

\bibitem{Cad} \textsc{Cadenillas, A., Hauam\'an-Aguilar, R.}\ (2016). \emph{Explicit Formula for the Optimal Government Debt Ceiling}. Ann.\ Oper.\ Res.\ \textbf{247(2)}, pp.\ 415--449.

\bibitem{Cad2} \textsc{Cadenillas, A., Huam\'an-Aguilar, R.}\ (2018). \emph{On the Failure to Reach the Optimal Government Debt Ceiling}. Risks \textbf{6(4)} 138, pp.\ 1--28.

\bibitem{CCF} \textsc{Callegaro, G., Ceci, C., Ferrari, G.}\ (2019). \emph{Optimal Reduction of Public Debt under Partial Observation of the Economic Growth}. \textbf{ArXiv}: 1901.08356.

\bibitem{DAuria} \textsc{D'Auria, B., Kella, O.}\ (2012). \emph{Markov modulation of a two-sided reflected Brownian motion with application to fluid queues.} Stoch.\ Proc.\ Appl.\ \textbf{122(4)}, pp.\ 1566--1581.

\bibitem{Ferrari} \textsc{Ferrari, G.}\ (2018). \emph{On the Optimal Management of Public Debt: a Singular Stochastic Control Problem}. SIAM J.\ Control Optim.\ \textbf{56(3)}, pp.\ 2036--2073.

\bibitem{FerrariYang} \textsc{Ferrari, G., Yang, S.}\ (2018). \emph{On an Optimal Extraction Problem with Regime Switching}. Adv.\ Appl.\ Probab.\ \textbf{50(3)}, pp. 671-705.

\bibitem{DeAFe14} \textsc{De Angelis, T., Ferrari, G.}~(2014). \emph{Stochastic Partially Reversible Investment Problem on a Finite Time-horizon: Free-boundary Analysis}. Stoch.~Process.~Appl.~\textbf{124}, pp.~4080--4119.

\bibitem{DM}\textsc{Dellacherie, C., Meyer, P.}\ $(1982)$. \emph{Probabilities and Potential A}. North-Holland Mathematics Studies $72$.

\bibitem{DixitPind} \textsc{Dixit, R., Pindyck, R.S.}~(1994). \emph{Investment under Uncertainty}. Princeton University Press. Princeton. 

\bibitem{EkstromPeskir} \textsc{Ekstr\"om, E., Peskir, G.}\ (2008). \emph{Optimal Stopping Games for Markov Processes}. SIAM J.~Control Optim.~\textbf{47(2)}, pp.\ 684--702.

\bibitem{FlemingSoner}\textsc{Fleming, W.H., Soner, H.M.}\ (2005). \emph{Controlled Markov processes and Viscosity Solutions}. 2nd Edition. Springer.

\bibitem{Guoetal} \textsc{Guo, X., Miao, J.J, Morellec, E.}\ $(2005)$. \emph{Irreversible Investment with Regime Shifts}. {J.\ Econ.\ Theory}\ \textbf{122(1)}, pp.\ 37--59.

\bibitem{GuoPham} \textsc{Guo, X., Pham, H.}\ $(2005)$. \emph{Optimal Partially Reversible Investment with Entry Decision and General Production Function}. Stoch.~Process.~Appl.~\textbf{115(5)}, pp.~705--736.

\bibitem{Pistorius}\textsc{Jiang, Z., Pistorius, M.}\ (2012). \emph{Optimal Dividend Distribution under Markov Regime Switching}. Finance Stoch.\ \textbf{16}, pp.\ 449--476.

\bibitem{KS84} \textsc{Karatzas, I., Shreve, S.E.}~(1984). \emph{Connections between Optimal Stopping and Singular Stochastic Control I. Monotone Follower Problems}. SIAM J.~Control Optim.~\textbf{22}, pp.\ 856--877.

\bibitem{KaratzasWang} \textsc{Karatzas, I., Wang, H.}~(2005). \emph{Connections between Bounded-Variation Control and Dynkin Games} in ``Optimal Control and Partial Differential Equations''; Volume in Honor of Professor Alain Bensoussan's 60th Birthday (J.L. Menaldi, A. Sulem and E. Rofman, eds.), pp.\ 353--362. IOS Press, Amsterdam.

\bibitem{Komlos} \textsc{Koml\'os, J.}~(1967). \emph{A generalization of a problem of Steinhaus}. Acta~Math.~Acad.~Sci.~Hungar.~\textbf{18}, pp.\ 217--229.

\bibitem{Meyer} \textsc{Meyer, P.A.}\ (1976). \emph{Lecture Notes in Mathematics 511}. Seminaire de Probabilities X, Universit\'e de Strasbourg. Springer-Verlag. New York.

\bibitem{IMF} \textsc{Ostry, J.D., Ghosh, A.R., Espinoza, R.}\ (2015). \emph{When Should Public Debt Be Reduced?}. IMF Staff Discussion Note SDN/15/10.

\bibitem{PeskirShir}\textsc{Peskir, G., Shiryaev, A.}~(2006). \emph{Optimal Stopping and Free-Boundary Problems}. Springer, Berlin.   

\bibitem{PeskirNash} \textsc{Peskir, G.}\ (2008). \emph{Optimal Stopping Games and Nash Equilibrium}. Theory Probab.\ Appl.\ \textbf{53(3)}, pp.\ 558--571.

\bibitem{RRR12} \textsc{Reinhart, C.M., Reinhart, V.R., Rogoff, K.S.}\ (2012). \emph{Debt Overhangs: Past and Present (No.\ w18015)}. National Bureau of Economic Research.

\bibitem{Shreveetal} \textsc{Shreve, S.E., Lehoczky, J.P., Gaver, D.P.}~(1984). \emph{Optimal Consumption for General Diffusions with Absorbing and Reflecting Barriers}. SIAM J.~Control Optim.~\textbf{22(1)}, pp.~55--75. 

\bibitem{Sk} \textsc{Skorokhod, A.V.}\ (1989). \emph{Asymptotic Methods in the Theory of Stochastic Differential Equations}. AMS, Providence, RI.

\bibitem{CadSot}\textsc{Sotomayor, L.R., Cadenillas, A.}\ $(2011)$. \emph{Classical and Singular Stochastic Control for the Optimal Dividend Policy when there is Regime Switching}. Ins.\ Math.\ Econ.\ \textbf{48}, pp.\ 344--354.

\bibitem{Zhou-Yin}\textsc{Zhu, C., Yin, G.}\ $(2009)$. \emph{On Strong Feller, Recurrence, and Weak Stabilization of Regime-Switching Diffusions}. {SIAM J.~Control Optim.} \textbf{48(3)}, pp.\ 2003--2031.

\end{thebibliography}
\end{document}